\documentclass[a4paper,12pt]{amsart}
\renewcommand{\o}{\vee}

\newcommand{\y}{\wedge}

\newcommand{\limp}{\rightarrow}

\usepackage[matrix,arrow,frame]{xy}
\newcommand{\alcentro}[1]{\ar @{-}[]+0;[#1]+0}
\newcommand{\bordes}[2][]{\ar @{--}@<#1 2ex>[]+0;[#2]+0}

\newcommand{\extremos}[2]{{\xy*+{#1}*\cir<2ex>{#2}\endxy}}
\newcommand{\singul}[1]{*++[o][F-]{#1}}
\newcommand{\sing}{*++[o][F-]{\bullet}}

\newcommand{\extrNEpi}[1][\bullet]{\extremos{#1}{ul^dr}}
\newcommand{\extrSWpi}[1][\bullet]{\extremos{#1}{dr^ul}}

\newcommand{\bordesup}{\bordes{r}}
\newcommand{\bordeinf}{\bordes[-]{r}}

\usepackage{amsmath}

\newcommand{\romb}{\diamond}

\newcommand{\0}{\Delta}

\DeclareMathOperator{\CON}{\mathrm{Con}}

\newcommand{\variedad}[1]{\mathcal{#1}}

\newcommand{\V}{\variedad{V}}

\newcommand{\id}{\approx}

\renewcommand{\id}{\approx}
\newcommand{\func}{\rightarrow}

\usepackage{amsthm}
\usepackage{verbatim}
\newlength{\ancho}
\textwidth=\ancho
\newcommand{\ATTACK}{$\boldsymbol{\forall}$}
\newcommand{\DEFENSE}{$\boldsymbol{\exists}$}
\newcommand{\cero}{\mathbf{0}}
\newcommand{\uno}{\mathbf{1}}
\newcommand{\masd}{+}

\newcommand{\maso}{+}
\newcommand{\poro}{*}
\newcommand{\pord}{*}

\newcommand{\qua}{'}
\newcommand{\til}{\tilde}
\newcommand{\impl}{\rightarrow}
\newcommand{\<}{(}
\renewcommand{\>}{)}
\newcommand{\ent}{\Rightarrow}
\newcommand{\tne}{\Leftarrow}

\renewcommand{\phi}{\varphi}
\newcommand{\phis}{{\varphi^*}}
\renewcommand{\th}{\theta}
\newcommand{\ths}{{\theta^*}}
\newcommand{\Cg}{\mathrm{Cg}}
\renewcommand{\romb}{\times}
\newcommand{\termmmth}{\sigma(\vec{X})}
\newcommand{\termmmths}{\sigma^*(\vec{X})}
\newcommand{\termmmphi}{\rho(\vec{X})}
\newcommand{\termmmphis}{\rho^*(\vec{X})}
\newcommand{\termuno}{\vec X}

\newcommand{\largo}[1]{|#1|}
\newtheorem{theorem}{Theorem}
\newtheorem{lemma}[theorem]{Lemma}
\newtheorem{prop}[theorem]{Proposition}
\newtheorem{corollary}[theorem]{Corollary}

\newtheorem{claim}{Claim}
\CompileMatrices

\title{Varieties with  Definable Factor Congruences}
\author{Pedro S\'{a}nchez Terraf
\quad  Diego J. Vaggione }
\thanks{Supported by CONICET and
    SECYT-UNC. 
}

\begin{document}
\begin{abstract}
We study direct product representations of  algebras in
varieties. We collect several
conditions expressing that these representations are \emph{definable}
in a first-order-logic sense, among them the concept of  Definable
Factor Congruences (DFC). The main results are that DFC  is a Mal'cev
property and that it  is equivalent to all other conditions
formulated; in particular we prove that $\V$ has DFC if  and only if 
$\V$ has $\vec{0}$ \& $\vec{1}$ and \emph{Boolean Factor Congruences}.  We also  obtain an explicit
 first order definition $\Phi$ of the kernel of the canonical projections via the terms 
associated to the Mal'cev condition for DFC, in such a manner it is preserved
by taking direct products and direct factors. The main tool is the use
of \emph{central elements,} which are a generalization of
both central idempotent elements in rings with identity and neutral
complemented elements in a bounded lattice.
\end{abstract}
\maketitle
\section{Introduction}
An \emph{algebra} is a nonempty set together with an arbitrary but
fixed collection of finitary operations. A \emph{variety} is an
equationally-definable class of algebras over the same language. 
A \emph{congruence} of an algebra $A$ is the kernel $\{\<a,b\>\in A : f(a) = f(b)\}$
 of a homomorphism $f$ with domain $A$; it is a \emph{factor congruence} of $A$ if
 $f$ is a projection onto a direct factor of $A$. Thus, a direct
 product representation is determined by the pair of
 \emph{complementary} factor congruences given by the canonical
 projections.

In this universal-algebraic setting, one key concept 
for the deeper study of direct product representations is that of \emph{central
  element}. This tool can be developed fruitfully in varieties
  with $\vec{0}$ \& $\vec{1}$, which we now define.

A \emph{variety with $\vec{0}$ \& $\vec{1}$} is a variety $\V$ in which
there exist unary terms $0_{1}(w),\dots,0_{l}(w),$ $1_{1}(w),\dots,1_{l}(w)$ such that
\begin{equation*}
\mathcal{V}\models \vec 0(w)=\vec 1(w)\rightarrow x=y,
\end{equation*}
where $w$, $x$ and $y$ are distinct variables, $\vec{0}=(0_{1},\dots,0_{l})$ and\ $\vec{1}=(1_{1},\dots,1_{l}).$

The  terms $\vec{0}$ and  $\vec{1}$  are  analog, in a rather general
manner, to identity
(top) and null (bottom) elements in rings (lattices), and its
existence in a variety is 
equivalent to the fact that no non-trivial algebra in the variety has a
trivial subalgebra. Throughout this paper we will assume that $\V$ is
a variety with $\vec{0}$ \& $\vec{1}$ such that the terms $\vec{0}$
and  $\vec{1}$ are closed. Of course, this can be
achieved when the language has a constant symbol and we will make this
assumption in order to  simplify and clarify our treatment. The
proofs remain valid in the general case. 

If $\vec a \in A^l$ and $\vec b \in B^l$, we will write $[\vec a, \vec b]$ in place
of $ ((a_{1},b_{1}),\dots,(a_{l},b_{l})) \in (A \times B)^l$.
If $A\in \mathcal{V}$, we say that $\vec{e}\in A^{l}$ is a \emph{central element} of $A$
if there exists an isomorphism $A\rightarrow A_{1}\times A_{2}$ such
that 
\[\vec{e}\mapsto [ \vec{0},\vec{1}].\]
Two central elements $\vec{e},\vec{f}$ will be called 
\emph{complementary} if there exists an isomorphism $ A\rightarrow A_{1}\times A_{2}$
such that $\vec{e}\mapsto [\vec{0},\vec{1}]$ and $\vec{f}%
\mapsto [\vec{1},\vec{0}].$ 

Central elements are a generalization of
both central idempotent elements in rings with identity and neutral
complemented elements in a bounded lattice. It is well known that in
the classical cases,  central elements  are powerful tools since they
translate the concepts of factor congruence and of direct product
representation into  first-order logic. In the general framework of an arbitrary variety $\V$ with $\vec{0}$ \&
$\vec{1}$, we may write this definability property in the following fashion:
\begin{quote}
 there exists a first order formula $\Phi(x,y,\vec z,\vec w)$ in the
  language of $\V$ such that for all $A, B\in \V$, and $a,c\in A$,
  $b,d\in B$,
  \[A\times B \models \Phi\bigl(\<a,b\>, \<c,d\>, [\vec 0, \vec 1],
  [\vec 1, \vec 0] \bigr) 
  \quad \text{ if and only if } \quad a=c.\]
\end{quote}

It was shown in~\cite{va4}  that in varieties  with
$\vec{0}$ \& $\vec{1}$ that have  the
Fraser-Horn-Hu property,  central elements have this
property, and the formula $\Phi$ may be chosen of the form $\exists
\bigwedge p=q$. Moreover, when the variety is congruence modular,
$\Phi$ can be chosen to be a conjunction of equations. One last family
of  varieties with $\vec{0}$ \& $\vec{1}$ (which generalize the above
mentioned examples) is that of varieties in which factor congruences
are \emph{compact}~\cite{9}.

Which is the most
general context in which central elements 
concentrate the information concerning the direct product
representations? One  answer is given by the following condition:
each pair of complementary central elements determines uniquely a pair of
complementary factor congruences associated to them. This property may
be written as:
\begin{quote}
 for $A\in \mathcal{V}$, the map
  \begin{quote}
    $(\theta ,\ths)\mapsto $ unique $(\vec{e},\vec{f})\in
    A^{l}\times A^{l}$ satisfying $\vec{0}\,\theta\, \vec{e}\,\ths\,\vec{1}$
    and $\vec{1}\,\theta\, \vec{f}\,\ths\,\vec{0}$
  \end{quote}
  is a bijection between the set of pairs of complementary factor congruences
  of $A$ and the set of pairs of complementary central elements of $A.$
\end{quote}



Another important concept involved in this paper is that of an algebra
with \emph{Boolean Factor Congruences (BFC)}, that is, an algebra in
which the set of factor congruences is a distributive sublattice of
its congruence lattice. Though it is not apparent, this concept is
intimately connected to the direct product construction. In
the classical work of Chang, J\'onsson and Tarski~\cite{ChaJonTar}
it is proved that BFC is equivalent to the \emph{strict refinement
 property} (a strengthening of the property that states that every two
direct product representations have a common refinement). It is also
noteworthy that in several works on sheaf representations
\cite{Burris, Comer, Davey, Krauss} BFC has
played an important role. For example, in Bigelow and
Burris~\cite{1} it is shown that in a variety with BFC the Boolean
product representations with directly indecomposable factors are
unique and coincide with the Pierce sheaf \cite{Pierce}. We refer the reader to the
work of Willard~\cite{7} as a key reference on BFC.

In the present work, we will prove that all the conditions mentioned
in this introduction are
indeed equivalent for the class of varieties with $\vec 0$ \& $\vec 1$:
\begin{theorem}\label{th:principal}
Let $\V$ be a variety with $\vec 0$ \& $\vec 1$. The following are
equivalent:
\begin{enumerate}
\item $\V$ has the \emph{Weak Determining Property}: for $A\in \mathcal{V}$, the map
  \begin{quote}
    $(\theta ,\ths)\mapsto $ unique $(\vec{e},\vec{f})\in
    A^{l}\times A^{l}$ satisfying $\vec{0}\,\theta\, \vec{e}\,\ths\,\vec{1}$
    and $\vec{1}\,\theta\, \vec{f}\,\ths\,\vec{0}$
  \end{quote}
  is a bijection between the set of pairs of complementary factor congruences
  of $A$ and the set of pairs of complementary central elements of $A.$
  
\item $\V$ has the \emph{Determining Property}: For $A\in
  \mathcal{V}$, the map 
  \begin{quote}
    $(\theta ,\ths)\mapsto $ unique $\vec{e}\in
    A^{l}$ satisfying $\vec{0}\,\theta\, \vec{e}\,\ths\,\vec{1}$
  \end{quote}
  is a bijection between the set of pairs of complementary factor congruences
  of $A$ and the set of  central elements of $A.$
\item $\V$ has \emph{Definable Factor Congruences (DFC)}: There exists a first order formula $\Phi(x,y,\vec z)$ in the
  language of $\V$ such that for all $A, B\in \V$, and $a,c\in A$,
  $b,d\in B$,
  \[A\times B \models \Phi\bigl(\<a,b\>, \<c,d\>, [\vec 0, \vec 1]\bigr)
  \quad \text{ if and only if } \quad a=c.\]
\item  There exists a first order formula $\Phi(x,y,\vec z,\vec w)$ in the
  language of $\V$ such that for all $A, B\in \V$, and $a,c\in A$,
  $b,d\in B$,
  \[A\times B \models \Phi\bigl(\<a,b\>, \<c,d\>, [\vec 0, \vec 1],
  [\vec 1, \vec 0] \bigr) 
  \quad \text{ if and only if } \quad a=c.\] 
\item $\V$ has BFC.
\end{enumerate}
Moreover, when the above equivalent conditions hold, the formula $\Phi$ in
(3) can be chosen to be preserved by direct products and direct
factors and for every $A\in \V$, the map
  \begin{quote}
    $\vec e \mapsto \Phi^A(\cdot, \cdot, \vec e)$
  \end{quote}
is a bijection between the set of central elements and the Boolean
algebra of factor
congruences of $A$.
\end{theorem}
We now briefly sketch the contents of each section. In
Section~\ref{sec:malc-cond-determ}  we
give a Mal'cev condition for an \emph{Left Determining Property} to be
defined there; this condition
is entirely analogous to a Mal'cev condition for BFC. The terms
obtained in this section are the building blocks for our definability
constructions. Section~\ref{sec:canonical-form-dfc} provides an
explicit formula $\Phi$ satisfying 
(3) of Theorem~\ref{th:principal}. This formula is constructed in such
a way that it is preserved by direct products and direct factors; this last
assertion is proved in the Appendix. In
Section~\ref{sec:centr-elem-vari}, we  characterize in 
first-order logic the (pairs of complementary) central elements in a variety with DFC and show that
the coordinates (in a direct product representation) of a central
element are central elements. Several results obtained in the previous
sections are compiled in Section~\ref{sec:main-theorem} to finish the
proof of the Main Theorem. Two examples are treated in
Section~\ref{sec:examples}. In~\ref{sec:01-does-not} we present  a
variety with $\vec 0$ \& $\vec 1$ that has not DFC; this shows that our
definitions do not
trivialize. In~\ref{sec:semil-expans}  we give an optimal formula for
the case of semilattices.

Throughout this paper the following notation will be used. For $A\in
\mathcal{V}$ and 
$\vec{a},\vec{b}\in A^{n}$,  
$\Cg  
^{A}(\vec{a},\vec{b})$ will denote the congruence generated by the set $%
\{(a_k,b_k):1\leq k\leq n\}$.  The symbols
$\nabla$ and $\Delta$ will stand for the universal and trivial
congruence, respectively. We will use  $\th \romb \ths = \0$ in
place of  ``$\th$ and $\ths$ are complementary factor
congruences''. The term algebra (in the language of $\V$) and the
$\V$-free algebra on $X$ will be 
denoted by $T(X)$ and $F(X)$, respectively.
The $i$-th component of an element $a$ in a direct product $\Pi_i A_i$
will be called $a^i = \pi_i(a)$; hence, if $a\in 
A_0\times A_1$,  $a = \<a^0,a^1\>$. If elements $a,b$ of an algebra
$A$ are related by a congruence $\th\in\CON A$, we will write
interchangeably $(a,b)\in \th$, $a \,\th\, b$ or $a
\stackrel{\th}{\equiv} b$. This notation generalizes
to tuples, viz.,  $\vec a
\,\th\, \vec b$ means $(a_i,b_i)\in \th$ for all $i$. 

\section{The Left Determining Property: a Mal'cev Condition}\label{sec:malc-cond-determ}
We will use the following \emph{Left Determining Property:}
\begin{quote}
For every $A\in \V$, $\vec e \in A^l$ and $\phi, \phis, \th, \ths\in
\CON A$, if $\phi \times \phis = \0$, $\th \times \ths = \0$, 
$\vec 0 \, \th \, \vec e \, \ths \, \vec 1$ and  $\vec 0 \, \phi \,
\vec e \,\phis \, \vec 1$, then $\th =\phi$.
\end{quote}

It is not difficult to see the implications Determining Property $\ent$ Left Determining Property $\ent$ Weak Determining Property.

The following theorem gives a Mal'cev condition for the Left Determining Property. 
Let  $s_i,t_i $ be $(2i+l)$-ary terms (in the language of
$\V$) for each 
$i=1,\dots,n$ and let $A$ an algebra in the language of $\V$ (not
necessarily in $\V$).  For
$(c,d,\vec e,a_1,b_1,\dots,a_n,b_n) \in A^{2+l+2n}$,
we  define  $\sigma(c,d,\vec e,a_1,b_1,\dots,a_n,b_n)$ to be the tuple
$(x,y,\vec z,x_1,y_1,\dots,x_n,y_n)$ given  by the following recursion:
\begin{align*}
  x&:= c  & x_j&:=s_j(x,y,\vec z,x_1,y_1,\dots,x_{j-1},y_{j-1}) \\
  y&:=c   & y_j &:= b_j \\
  \vec z&:=\vec 0
\end{align*}
We define $\sigma^*$, $\rho$, $\rho^*$ analogously.
\begin{itemize}
\item  $\sigma^*(c,d,\vec e,a_1,b_1,\dots,a_n,b_n) :=
  (x,y,\vec z,x_1,y_1,\dots,x_n,y_n)$ where:
  \begin{align*}
    x&:= c  & x_j&:=t_j(x,y,\vec z,x_1,y_1,\dots,x_{j-1},y_{j-1}) \\
    y&:=d   & y_j &:= b_j \\
    \vec z&:=\vec 1
  \end{align*}
\item  $\rho(c,d,\vec e,a_1,b_1,\dots,a_n,b_n) :=
  (x,y,\vec z,x_1,y_1,\dots,x_n,y_n)$ where:
\begin{align*}
  x&:= c  & x_j &:= a_j \\
  y&:=d   & y_j&:=s_j(x,y,\vec z,x_1,y_1,\dots,x_{j-1},y_{j-1}) \\
 \vec z&:=\vec 0
\end{align*}
\item  $\rho^*(c,d,\vec e,a_1,b_1,\dots,a_n,b_n) :=
  (x,y,\vec z,x_1,y_1,\dots,x_n,y_n)$ where:
  \begin{align*}
    x&:= c  & x_j &:= a_j \\
    y&:=d   & y_j&:=t_j(x,y,\vec z,x_1,y_1,\dots,x_{j-1},y_{j-1}) \\
    \vec z&:= \vec 1
  \end{align*}
\end{itemize}

We first state without proof a lemma concerning these functions.
\begin{lemma}\label{l:igualdad_congr}
For every $c,d,\vec e,a_1,b_1,\dots,a_n,b_n \in A$, we have the following
identities:
\begin{multline}
 \Cg(c,d) \o \Cg(\vec e,\vec 0) \o \bigvee_i
\Cg(a_i,s_i(c,d,\vec e,a_1,b_1,\dots,a_{i-1},b_{i-1}))= \\
 = \Cg((c,d,\vec e,a_1,b_1,\dots,a_n,b_n),\sigma(c,d,\vec e,a_1,b_1,\dots,a_n,b_n))
\end{multline}
\begin{align*}
\Cg(\vec e,\vec 1) \o \bigvee_i \Cg(a_i,t_i(c,d,\vec e,\dots,a_{i-1},b_{i-1})) &=
\Cg((c,d,\vec e,\dots,a_n,b_n),\sigma^*(c,d,\vec e,\dots,a_n,b_n))\\
\Cg(\vec e,\vec 0) \o \bigvee_i \Cg(b_i,s_i(c,d,\vec e,\dots,a_{i-1},b_{i-1})) & =
\Cg((c,d,\vec e,\dots,a_n,b_n),\rho(c,d,\vec e,\dots,a_n,b_n))\\
\Cg(\vec e,\vec 1) \o \bigvee_i  \Cg(b_i,t_i(c,d,\vec e,\dots,a_{i-1},b_{i-1})) & =
  \Cg((c,d,\vec e,\dots,a_n,b_n),\rho^*(c,d,\vec e,\dots,a_n,b_n))
\end{align*}
\end{lemma}

In the proofs that follow, we will repeatedly find elements in an
 algebra that solve congruential ``equations'' of the form
\[a \ \stackrel{\th}{\equiv} \  x  \ \stackrel{\ths}{\equiv} \ b\]
when $\th\romb\ths = \0$. Using the functions $\sigma$, $\sigma^*$,
$\rho$ and  $\rho^*$  just defined, we can
 assert conclusions from the way elements like $x$ are
 constructed. This is the content of the following immediate consequences
 of Lemma~\ref{l:igualdad_congr}:

\begin{corollary}\label{l:recursion_aes}
Given $c,d,\vec e\in A$ and $\th,\ths\in\CON A$
such that  $\vec 0 \, \th \, \vec e \, \ths \, \vec 1$ and $c \, \th \, d $, 
for every $a_i$ and $b_i$ with $i=1,\dots,n$ such that
 \begin{equation*}
\begin{split}
s_1(c,d,\vec e)\stackrel{\th}{\equiv} &\  a_1 \stackrel{\ths}{\equiv}
t_1(c,d,\vec e)\\ 
s_2(c,d,\vec e,a_1,b_1)\stackrel{\th}{\equiv} &\  a_2 \stackrel{\ths}{\equiv}
 t_2(c,d,\vec e,a_1,b_1)\\ 
& \dots \\
s_{i+1}(c,d,\vec e,a_1,b_1,\dots,a_{i},b_{i})\stackrel{\th}{\equiv} & \
a_{i+1}
\stackrel{\;\ths}{\equiv}t_{i+1}(c,d,\vec e,a_1,b_1,\dots,a_{i},b_{i}) 
\end{split}
\end{equation*}
we have
\begin{equation}\label{eq:4}
t(\sigma(c,d,\vec e,a_1,b_1,\dots,a_n,b_n)) 
\stackrel{\th}{\equiv}
t(c,d,\vec e,a_1,b_1,\dots,a_n,b_n)  \stackrel{\ths}{\equiv}
t(\sigma^*(c,d,\vec e,a_1,b_1,\dots,a_n,b_n)) 
\end{equation}
for every $(2n+l+2)$-ary term $t$ in the language of $\V$.
\end{corollary}
The next result is entirely analogous.
\begin{corollary}\label{l:recursion_bes}
Suppose $c,d,\vec e\in A$ and $\phi,\phis\in\CON A$
such that  $\vec 0 \, \phi \,
\vec e \,\phis \, \vec 1$. If $a_i$ and $b_i$ satisfy
\begin{equation*}
\begin{split}
s_1(c,d,\vec e)\stackrel{\phi}{\equiv} &\  b_1 \stackrel{\phis}{\equiv}
t_1(c,d,\vec e)\\ 
& \dots \\
s_{i+1}(c,d,\vec e,a_1,b_1,\dots,a_{i},b_{i})\stackrel{\phi}{\equiv} & \ b_{i+1}
 \stackrel{\;\phis}{\equiv}t_{i+1}(c,d,\vec e,a_1,b_1,\dots,a_{i},b_{i}) 
\end{split}
\end{equation*}
we obtain
\begin{equation}\label{eq:5}
 t(\rho(c,d,\vec e,a_1,b_1,\dots,a_n,b_n)) \stackrel{\phi}{\equiv}
 t(c,d,\vec e,a_1,b_1,\dots,a_n,b_n)  \stackrel{\phis}{\equiv}
 t(\rho^*(c,d,\vec e,a_1,b_1,\dots,a_n,b_n))
\end{equation}
for every $(2n+l+2)$-ary term $t$ in the language of $\V$. 
\end{corollary}
We will also need the following (Gr\"{a}tzer's)
version of  Mal'cev's key observation on principal congruences.

\begin{lemma}\label{malsev}Let $A$ be any algebra and let $a,b\in A,$
 $\vec{a},\vec{b}\in
A^{n}.$ Then $(a,b)\in \Cg ^{A}(\vec{a},\vec{b})$ if and only if 
there
exist $(n+m)$-ary terms 
$p_{1}(\vec{x},\vec{u}),\dots,p_{k}(\vec{x},\vec{u})$,
with $k$ odd and, $\vec{u}\in A^{m}$ such that: 
\begin{equation*}
\begin{split}
a& =p_{1}(\vec{a},\vec{u}) \\
p_{i}(\vec{b},\vec{u})& =p_{i+1}(\vec{b},\vec{u}),i\text{
odd} \\
p_{i}(\vec{a},\vec{u})& =p_{i+1}(\vec{a},\vec{u}),i\text{
even} \\
p_{k}(\vec{b},\vec{u})& =b
\end{split}
\end{equation*}
\end{lemma}
We will use $\largo{\alpha}$ to denote the length of a word $\alpha$
and $\varepsilon$ will denote the empty word.
\begin{theorem}\label{th:malcev_dfc}
Let $\V$ a variety with $\vec 0$ \& $\vec 1$. $\V$ has Left Determining Property if and only
if there exist  integers 
$N=2k$  and $n$, $(2i+l)$-ary terms $s_i$ and $t_i$ for each
$i=1,\dots,n$, and for every word  $\alpha$ in the alphabet $\{1,\dots,N\}$ of
length no greater than $N$ there are terms $L_\alpha, R_\alpha $ such that 
\medskip

\noindent\fbox{$\largo{\alpha}=N$}
\begin{equation}
\begin{split}\label{eq:dfc_N}
L_\alpha(\termmmphi) & \id R_\alpha(\termmmphi) \\
L_\alpha(\termmmphis) & \id R_\alpha(\termmmphis) 
\end{split}
\end{equation}	
\noindent\fbox{$\largo{\alpha}=0$}
\begin{equation}\label{eq:1}
\begin{split}
x   \id L_\varepsilon(x,y,\vec z,x_1,y_1,\dots,x_n,y_n) \\
R_\varepsilon(x,y,\vec z,x_1,y_1,\dots,x_n,y_n)  \id y
\end{split}
\end{equation}

\begin{align}
L_\varepsilon(\termmmphi) & \id L_{1}(\termmmphi)\label{eq:dfc_0_phi_1}\\
R_{j}(\termmmphi) & \id L_{j+1}(\termmmphi) \qquad \text{ if
$1\leq j \leq N-1$} \label{eq:dfc_0_phi}\\
R_{ N}(\termmmphi) & \id R_{\varepsilon}(\termmmphi) \label{eq:dfc_0_phi_N}
\end{align}

\noindent\fbox{$0<\largo{\alpha}<N$}
\smallskip

 If $\largo{\alpha}$ is even then
\begin{align}
L_\alpha(\termmmphi) & \id L_{\alpha 1}(\termmmphi)\label{eq:dfc_par_phi_1} \\
R_{\alpha j}(\termmmphi) & \id L_{\alpha (j+1)}(\termmmphi)\label{eq:dfc_par_phi} \qquad
\text{ if $1\leq j \leq k-1$} \\
R_{\alpha k}(\termmmphi) & \id R_{\alpha}(\termmmphi)\label{eq:dfc_par_phi_k}\\
\begin{split}\label{eq:dfc_par_phi*}
L_\alpha(\termmmphis) & \id L_{\alpha (k+1)}(\termmmphis) \\
R_{\alpha j}(\termmmphis) & \id L_{\alpha (j+1)}(\termmmphis) \qquad
\text{ if $k+1\leq j \leq N-1$} \\
R_{\alpha N}(\termmmphis) & \id R_{\alpha}(\termmmphis)
\end{split}
\end{align}

 If $\largo{\alpha}$ is odd then
\begin{align}
\begin{split}
L_\alpha(\termmmth) & \id L_{\alpha 1}(\termmmth)\label{eq:2} \\
R_{\alpha j}(\termmmth) & \id L_{\alpha (j+1)}(\termmmth) 
\qquad
\text{ if $1\leq j \leq k-1$} \\
R_{\alpha k}(\termmmth) & \id R_{\alpha}(\termmmth)
\end{split}\\
\begin{split}
L_\alpha(\termmmths) & \id L_{\alpha (k+1)}(\termmmths)\label{eq:3} \\
R_{\alpha j}(\termmmths) & \id L_{\alpha
(j+1)}(\termmmths)
\qquad \text{ if $k+1\leq j \leq N-1$} \\
R_{\alpha N}(\termmmths) & \id R_{\alpha}(\termmmths)
\end{split}
\end{align}
where $\vec X =
(x,y,\vec z,x_1,y_1,\dots, x_n,y_n)$ and  $\sigma$, $\sigma^*$,
$\rho$ and  $\rho^*$ are defined relative to  $s_i,$ $t_i$,
on $T(\vec X)$.
\end{theorem}

\begin{proof}
($\tne$) Assume the existence of the terms, and suppose $\phi \romb 
\phis= \0$, $\th \romb \ths = \0$, $\vec 0 \, \th \, e \, \ths \, \vec 1$, $\vec 0 \, \phi \,
e \,\phis \, \vec 1$, and $c \, \th \, d $. We want to see $c \, \phi \, d$. There exist unique $a_i,b_i$ satisfying the following relations:
\begin{equation*}\label{eq:recursion}
\begin{split}
s_1(c,d,\vec e)\stackrel{\th}{\equiv} &\  a_1 \stackrel{\ths}{\equiv}
t_1(c,d,\vec e)\\ 
s_1(c,d,\vec e)\stackrel{\phi}{\equiv} &\  b_1 \stackrel{\phis}{\equiv}
t_1(c,d,\vec e)\\ 
& \dots \\
s_{j+1}(c,d,\vec e,a_1,b_1,\dots,a_{j},b_{j})\stackrel{\th}{\equiv} & \
a_{j+1}
\stackrel{\;\ths}{\equiv}t_{j+1}(c,d,\vec e,a_1,b_1,\dots,a_{j},b_{j}) \\
s_{j+1}(c,d,\vec e,a_1,b_1,\dots,a_{j},b_{j})\stackrel{\phi}{\equiv} & \ b_{j+1}
\stackrel{\;\phis}{\equiv}t_{j+1}(c,d,\vec e,a_1,b_1,\dots,a_{j},b_{j}) 
\end{split}
\end{equation*}
Note that their definition combines schemes in
Corollaries~\ref{l:recursion_aes} and~\ref{l:recursion_bes}. So, by
equations~(\ref{eq:4}) and (\ref{eq:5}) we have, taking $t:=L_\alpha, R_\alpha$:
\begin{gather}
\begin{split}\label{eq:6}
L_\alpha(\sigma(c,d,\vec e,a_1,b_1,\dots,a_n,b_n)) &\stackrel{\th}{\equiv}
L_\alpha(c,d,\vec e,a_1,b_1,\dots,a_n,b_n)  \stackrel{\ths}{\equiv}
L_\alpha(\sigma^*(c,d,\vec e,a_1,b_1,\dots,a_n,b_n)) \\
L_\alpha(\rho(c,d,\vec e,a_1,b_1,\dots,a_n,b_n)) &\stackrel{\phi}{\equiv}
L_\alpha(c,d,\vec e,a_1,b_1,\dots,a_n,b_n)  \stackrel{\phis}{\equiv}
L_\alpha(\rho^*(c,d,\vec e,a_1,b_1,\dots,a_n,b_n))
\end{split}	\\
\begin{split}\label{eq:36}
R_\alpha(\sigma(c,d,\vec e,a_1,b_1,\dots,a_n,b_n)) &\stackrel{\th}{\equiv}
R_\alpha(c,d,\vec e,a_1,b_1,\dots,a_n,b_n)  \stackrel{\ths}{\equiv}
R_\alpha(\sigma^*(c,d,\vec e,a_1,b_1,\dots,a_n,b_n)) \\
R_\alpha(\rho(c,d,\vec e,a_1,b_1,\dots,a_n,b_n)) &\stackrel{\phi}{\equiv}
R_\alpha(c,d,\vec e,a_1,b_1,\dots,a_n,b_n)  \stackrel{\phis}{\equiv}
R_\alpha(\rho^*(c,d,\vec e,a_1,b_1,\dots,a_n,b_n))
\end{split}	
\end{gather}
for every $\alpha$. We'll prove  inductively  that 
\begin{equation}\label{eq:37}
L_\alpha(c,d,\vec e,a_1,b_1,\dots,a_n,b_n)=R_\alpha(c,d,\vec
e,a_1,b_1,\dots,a_n,b_n)
\end{equation}
for all $\alpha\neq\varepsilon$. Take $\alpha$ such that $\largo{\alpha}=N$,
then
\begin{align*}
L_\alpha(c,d,\vec e,a_1,b_1,\dots,a_n,b_n)  & \stackrel{\phi}{\equiv} 
L_\alpha(\rho(c,d,\vec e,a_1,b_1,\dots,a_n,b_n))
&& \text{by equations~(\ref{eq:6})}\\
&=  R_{\alpha}(\rho(c,d,\vec e,a_1,b_1,\dots,a_n,b_n)) &&  \text{using identities~(\ref{eq:dfc_N})}\\
&  \stackrel{\phi}{\equiv}  R_{\alpha}(c,d,\vec e,a_1,b_1,\dots,a_n,b_n) && \text{by
equations~(\ref{eq:36})}\\
\intertext{And,}
L_\alpha(c,d,\vec e,a_1,b_1,\dots,a_n,b_n)  & \stackrel{\;\phis}{\equiv} 
L_\alpha(\rho^*(c,d,\vec e,a_1,b_1,\dots,a_n,b_n))
&& \text{by equations~(\ref{eq:6})}\\
&=  R_{\alpha}(\rho^*(c,d,\vec e,a_1,b_1,\dots,a_n,b_n)) &&  \text{using identities~(\ref{eq:dfc_N})}\\
&  \stackrel{\;\phis}{\equiv}  R_{\alpha}(c,d,\vec e,a_1,b_1,\dots,a_n,b_n) && \text{by
equations~(\ref{eq:36})}
\end{align*} 
Hence $\bigl(L_\alpha(c,d,\vec e,a_1,b_1,\dots,a_{n},b_{n}),R_\alpha(c,d,\vec e,a_1,b_1,\dots,a_{n},b_{n})\bigr)\in
\phi \cap \phis = \0$ and then $L_\alpha(c,d,\vec e,a_1,b_1,\dots,a_{n},b_{n})=R_\alpha(c,d,\vec e,a_1,b_1,\dots,a_{n},b_{n}).$

Take $\alpha\neq\varepsilon$ 
of odd length and assume 
\[L_{\alpha
j}(c,d,\vec e,a_1,b_1,\dots,a_{n},b_{n}) = R_{\alpha j}(c,d,\vec
e,a_1,b_1,\dots,a_{n},b_{n})\]
 for every
$j=1,\dots,N$. We check that 
\[L_\alpha(c,d,\vec e,a_1,b_1,\dots,a_{n},b_{n}) \stackrel{\th}{\equiv}  R_\alpha(c,d,\vec e,a_1,b_1,\dots,a_{n},b_{n})\]
\begin{align*}
L_\alpha(c,d,\vec e,a_1,b_1,\dots,a_n,b_n)  & \stackrel{\th}{\equiv} 
L_\alpha(\sigma(c,d,\vec e,a_1,b_1,\dots,a_n,b_n))
&& \text{by equations~(\ref{eq:6})} \\
&= L_{\alpha 1}(\sigma(c,d,\vec e,a_1,b_1,\dots,a_n,b_n))
&& \text{by identities~(\ref{eq:2})} \\
 & \stackrel{\th}{\equiv}  L_{\alpha 1}(c,d,\vec e,a_1,b_1,\dots,a_n,b_n) &&
\text{by equations~(\ref{eq:6})} \\
& = R_{\alpha 1}(c,d,\vec e,a_1,b_1,\dots,a_n,b_n)
&& \text{by inductive hypothesis} \\
 & \stackrel{\th}{\equiv}  R_{\alpha 1}(\sigma(c,d,\vec e,a_1,b_1,\dots,a_n,b_n))
&& \text{by equations~(\ref{eq:36})}  \\
 & \stackrel{\th}{\equiv}  \ \cdots && \text{using~(\ref{eq:2}) and iterating}\\ 
& = R_{\alpha k}(\sigma(c,d,\vec e,a_1,b_1,\dots,a_n,b_n)) \\
& = R_{\alpha}(\sigma(c,d,\vec e,a_1,b_1,\dots,a_n,b_n)) && \text{using
identities~(\ref{eq:2})}\\
&  \stackrel{\th}{\equiv}  R_{\alpha} (c,d,\vec e,a_1,b_1,\dots,a_n,b_n),
\end{align*} 
The same way for $\ths$:
\begin{align*}
L_\alpha(c,d,\vec e,a_1,b_1,\dots,a_n,b_n)  & \stackrel{\ths}{\equiv} 
L_\alpha(\sigma^*(c,d,\vec e,a_1,b_1,\dots,a_n,b_n))
&& \text{by equations~(\ref{eq:6})}\\
&= L_{\alpha (k+1)}(\sigma^*(c,d,\vec e,a_1,b_1,\dots,a_n,b_n))
&& \text{by identities~(\ref{eq:3})} \\
 & \stackrel{\ths}{\equiv}  L_{\alpha (k+1)}(c,d,\vec e,a_1,b_1,\dots,a_n,b_n) &&
\text{by equations~(\ref{eq:6})} \\
& = R_{\alpha (k+1)}(c,d,\vec e,a_1,b_1,\dots,a_n,b_n)
&& \text{by inductive hypothesis} \\
 & \stackrel{\ths}{\equiv}  R_{\alpha (k+1)}(\sigma^*(c,d,\vec e,a_1,b_1,\dots,a_n,b_n))
&& \text{by equations~(\ref{eq:36})}\\
 & \stackrel{\ths}{\equiv}  \ \cdots && \text{using~(\ref{eq:3}) and iterating\dots}\\
& = R_{\alpha N}(\sigma^*(c,d,\vec e,a_1,b_1,\dots,a_n,b_n)) \\
& = R_{\alpha}(\sigma^*(c,d,\vec e,a_1,b_1,\dots,a_n,b_n)) && \text{using
identities~(\ref{eq:3})}\\
&  \stackrel{\ths}{\equiv}  R_{\alpha} (c,d,\vec e,a_1,b_1,\dots,a_n,b_n)&& \text{by equations~(\ref{eq:36})}
\end{align*} 
Hence $\bigl(L_\alpha(c,d,\vec e,a_1,b_1,\dots,a_{n},b_{n}),R_\alpha(c,d,\vec e,a_1,b_1,\dots,a_{n},b_{n})\bigr)\in
\th\cap\ths = \0$, and therefore they are equal. 

If $\alpha\neq\varepsilon$ has even length,
\begin{align*}
L_\alpha(c,d,\vec e,a_1,b_1,\dots,a_n,b_n)  & \stackrel{\phi}{\equiv} 
L_\alpha(\rho(c,d,\vec e,a_1,b_1,\dots,a_n,b_n))
&& \text{by equations~(\ref{eq:6})}\\
&= L_{\alpha 1}(\rho(c,d,\vec e,a_1,b_1,\dots,a_n,b_n))
&& \text{by identity~(\ref{eq:dfc_par_phi_1})} \\
 & \stackrel{\phi}{\equiv}  L_{\alpha 1}(c,d,\vec e,a_1,b_1,\dots,a_n,b_n) &&
\text{by equations~(\ref{eq:6})} \\
& = R_{\alpha 1}(c,d,\vec e,a_1,b_1,\dots,a_n,b_n)
&& \text{by inductive hypothesis} \\
 & \stackrel{\phi}{\equiv}  R_{\alpha 1}(\rho(c,d,\vec e,a_1,b_1,\dots,a_n,b_n))
&& \text{by equations~(\ref{eq:36})}\\
 & \stackrel{\phi}{\equiv}  \ \cdots && \text{using~(\ref{eq:dfc_par_phi}) and iterating\dots}\\
& = R_{\alpha k}(\rho(c,d,\vec e,a_1,b_1,\dots,a_n,b_n)) \\
& = R_{\alpha}(\rho(c,d,\vec e,a_1,b_1,\dots,a_n,b_n)) && \text{using
identity~(\ref{eq:dfc_par_phi_k})}\\
&  \stackrel{\phi}{\equiv}  R_{\alpha} (c,d,\vec e,a_1,b_1,\dots,a_n,b_n)&& \text{by equations~(\ref{eq:36})}
\end{align*} 
proves $\bigl(L_\alpha(c,d,\vec e,a_1,b_1,\dots,a_{n},b_{n}),R_\alpha(c,d,\vec e,a_1,b_1,\dots,a_{n},b_{n})\bigr)\in
\phi$, and
\begin{align*}
L_\alpha(c,d,\vec e,a_1,b_1,\dots,a_n,b_n)  & \stackrel{\;\phis}{\equiv} 
L_\alpha(\rho^*(c,d,\vec e,a_1,b_1,\dots,a_n,b_n))
&& \text{by equations~(\ref{eq:6})}\\
&= L_{\alpha (k+1)}(\rho^*(c,d,\vec e,a_1,b_1,\dots,a_n,b_n))
&& \text{by identities~(\ref{eq:dfc_par_phi*})} \\
 & \stackrel{\;\phis}{\equiv}  L_{\alpha (k+1)}(c,d,\vec e,a_1,b_1,\dots,a_n,b_n) &&
\text{by equations~(\ref{eq:6})} \\
& = R_{\alpha (k+1)}(c,d,\vec e,a_1,b_1,\dots,a_n,b_n)
&& \text{by inductive hypothesis} \\
 & \stackrel{\;\phis}{\equiv}  R_{\alpha (k+1)}(\rho^*(c,d,\vec e,a_1,b_1,\dots,a_n,b_n))
&& \text{by equations~(\ref{eq:36})}\\
 & \stackrel{\;\phis}{\equiv}  \ \cdots && \text{using~(\ref{eq:dfc_par_phi*}) and iterating\dots}\\
& = R_{\alpha N}(\rho^*(c,d,\vec e,a_1,b_1,\dots,a_n,b_n)) \\
& = R_{\alpha}(\rho^*(c,d,\vec e,a_1,b_1,\dots,a_n,b_n)) && \text{using
identities~(\ref{eq:dfc_par_phi*})}\\
&  \stackrel{\;\phis}{\equiv}  R_{\alpha} (c,d,\vec e,a_1,b_1,\dots,a_n,b_n)&& \text{by equations~(\ref{eq:36})}
\end{align*} 
completes this case. Finally, we have:
\begin{align*}
c &=L_\varepsilon(\rho(c,d,\vec e,a_1,b_1,\dots,a_n,b_n)) &&  \text{using identities~(\ref{eq:1})}\\ 
&= L_{ 1}(\rho(c,d,\vec e,a_1,b_1,\dots,a_n,b_n))
&& \text{by identity~(\ref{eq:dfc_0_phi_1})} \\
 & \stackrel{\phi}{\equiv}  L_{ 1}(c,d,\vec e,a_1,b_1,\dots,a_n,b_n) &&
\text{by equations~(\ref{eq:6})}\\
& = R_{ 1}(c,d,\vec e,a_1,b_1,\dots,a_n,b_n)
&& \text{by~(\ref{eq:37})} \\
 & \stackrel{\phi}{\equiv}  R_{ 1}(\rho(c,d,\vec e,a_1,b_1,\dots,a_n,b_n))
&& \text{by equations~(\ref{eq:36})}\\
 & \stackrel{\phi}{\equiv}  \ \cdots && \text{using equations~(\ref{eq:dfc_0_phi}) and iterating\dots}\\
& = R_{N}(\rho(c,d,\vec e,a_1,b_1,\dots,a_n,b_n)) \\
& = R_\varepsilon(\rho(c,d,\vec e,a_1,b_1,\dots,a_n,b_n)) && \text{using
identity~(\ref{eq:dfc_0_phi_N})}\\
& = d  &&  \text{using identities~(\ref{eq:1})}
\end{align*} 

This proves $(c,d)\in\phi$.

\noindent ($\ent$) For each set of variables $Y$, define
\begin{align*}
Y^* &:= Y \cup \{x_{p,q} : p,q \in T(Y)\} \cup \{y_{p,q} : p,q \in
T(Y)\}  \\
Y^{0*} &:= Y \\
Y^{(n+1)*} &:= (Y^{n*})^* \\
Y^\infty &:= \bigcup_{n\geq 1} Y^{n*}
\end{align*}
where $x_{p,q}$ and $y_{p,q}$ are new variables. Take $Z:=\{x,y,z_1,\dots,z_l\}$
and $F:=F(Z^\infty)$. Define the \emph{index} of $p\in T(Z^\infty)$ as $ind(p) = \min\{j:
p\in T(Z^{j*})\}$; it is evident that if $ind(x_{p,q})\leq
ind(x_{r,s})$, neither $p$ nor $q$ can be terms depending on
$x_{r,s}$. The same holds for $ind(x_{p,q})\leq
ind(y_{r,s})$ and symmetrically, and for  $ind(y_{p,q})\leq
ind(y_{r,s})$. 

Take the following congruences  on $F$:
\begin{align*}
\th &:= \Cg(\vec 0,\vec z) \o \Cg(x,y) \o \bigvee\{\Cg(p,x_{p,q}): p,q \in F\} & \delta_0 & = \epsilon_0 := \0^F \\
\ths &:= \Cg(\vec 1,\vec z) \o \bigvee\{\Cg(x_{p,q},q): p,q \in F\} &
\delta_{n+1} &:= (\th \o \epsilon_n)\cap (\ths \o \epsilon_n)\\ 
\phi &:= \Cg(\vec 0,\vec z)  \o \bigvee\{\Cg(p,y_{p,q}): p,q \in
F\}&\epsilon_{n+1} &:= (\phi \o \delta_n)\cap (\phis \o \delta_n)
\\ 
\phis &:= \Cg(\vec 1,\vec z) \o \bigvee\{\Cg(y_{p,q},q): p,q \in
F\}&\delta_\infty &:= \bigvee_{n\geq 0} \delta_n =  \bigvee_{n\geq 0} \epsilon_n
\end{align*}
By construction, $\phi \circ \phis = \th \circ \ths = \nabla^F$,
$\vec 0 \, \th \, \vec z \, \ths \, \vec 1$, $\vec 0 \, \phi \, \vec z \,\phis
\, \vec 1$, and $x \, \th \, y $. 
Observe that if $(a,b)\in (\phi\o\delta_\infty) \cap
 (\phis\o\delta_\infty)$ then there exists an $n\geq 0$ such that
$(a,b)\in (\phi\o\delta_n) \cap (\phis\o\delta_n)$. But this
congruence is exactly $\epsilon_{n+1}$, hence $(a,b) \in
\epsilon_{n+1} \subseteq \delta_\infty$. We may conclude $(\phi\o\delta_\infty) \cap
 (\phis\o\delta_\infty) = \delta_\infty $. The same happens with
 $\theta$ and $\ths$,
hence 
\[(\phi\o\delta_\infty)/\delta_\infty \romb 
(\phis\o\delta_\infty)/\delta_\infty = \0 \qquad (\th\o\delta_\infty)/\delta_\infty \romb
(\ths\o\delta_\infty)/\delta_\infty = \0\]
in $F/\delta_\infty$. Then, by Left Determining Property we have  $(x/\delta_\infty,y/\delta_\infty) \in
(\phi\o\delta_\infty)/\delta_\infty$ and hence  $(x,y) \in
\phi \o \delta_\infty$. We may find an even integer $N=2k$ such that  $(x,y) \in
\phi \circ^{2N} \delta_N^N$, where $\delta_N^N$ is the result of replacing each
occurrence of ``$\o$'' in the definition of $\delta_N$ by $\circ^N$,
the $n$-fold relational product. 

We will inductively define terms $L_\alpha$ and $R_\alpha$, for $\alpha$ a word of
length at most $N$ in the alphabet $\{1,\dots,N\}$ such that:
\begin{align}
x & = L_\varepsilon  &    y = R_\varepsilon \label{eq:primera_ec_congruencias}\\ 
	(L_\varepsilon,L_{ 1}) &\in\phi  &  (R_{ N},R_\varepsilon) \in\phi\label{eq:fin_primer_grupo}  \\
(L_{ i},R_{ i}) &\in\delta_{N}^N    \qquad  \text{ if $1\leq i \leq
N$}\label{eq:29} \\
(R_{ i},L_{ (i+1)}) &\in\phi   \qquad   \text{ if $1\leq i \leq N-1$}\label{eq:principio_seg_grupo}.
\end{align}
For $\alpha\neq\varepsilon$ with $\largo{\alpha}<N$ an odd integer, 
\begin{align}
	(L_\alpha,L_{\alpha 1}) &\in\th  &  (R_{\alpha k},R_\alpha)
        &\in\th \label{eq:25}\\ 
(L_\alpha,L_{\alpha (k+1)}) &\in\ths  &  (R_{\alpha N},R_\alpha)
&\in\ths \label{eq:26}\\
(L_{\alpha i},R_{\alpha i}) &\in\epsilon_{N-\largo{\alpha}}^N \label{eq:24} \quad \text{ if $1\leq i \leq
N$} \\
(R_{\alpha i},L_{\alpha (i+1)}) &\in\th  \qquad  \text{ if $1\leq i \leq
  k-1$}\label{eq:27}\\ 
(R_{\alpha i},L_{\alpha (i+1)}) &\in\ths   \qquad \text{ if $k+1\leq i \leq N-1$}\label{eq:28}
\end{align}
and for $\alpha\neq\varepsilon$ with $\largo{\alpha}<N$ an even integer:
\begin{align}
	(L_\alpha,L_{\alpha 1}) &\in\phi  &  (R_{\alpha k},R_\alpha) \in\phi \label{eq:8}\\
(L_\alpha,L_{\alpha (k+1)}) &\in\phis  &  (R_{\alpha N},R_\alpha) \in\phis \label{eq:fin_ter_grupo}\\
(L_{\alpha i},R_{\alpha i}) &\in\delta_{N-\largo{\alpha}}^N \quad   \text{ if $1\leq i \leq
N$}
  \label{eq:7} \\
(R_{\alpha i},L_{\alpha (i+1)}) &\in\phi  \qquad  \text{ if $1\leq i \leq k-1$}\label{eq:penultima}\\
(R_{\alpha i},L_{\alpha (i+1)}) &\in\phis   \qquad   \text{ if $k+1\leq i
  \leq N-1$}\label{eq:ultima_ec_congruencias}
\end{align}

We take $ L_\varepsilon:=x$ and $R_\varepsilon:=y$. Since we know $(x,y) \in
\phi \circ^{2N} \delta_N^N$, we define  $L_i$,
$R_i$ for $i=1,\dots,N$ to be terms satisfying 
\begin{equation*}
x \; \phi \; L_1 \;\delta_N^N\; R_1 \; \phi \; L_{2} \;\delta_N^N
\cdots \phi \;  L_N \;\delta_N^N\; R_N \;\phi \; y 
\end{equation*}
Note that these terms
satisfy~(\ref{eq:primera_ec_congruencias})--(\ref{eq:ultima_ec_congruencias})
whenever they can be checked.

Suppose we have defined the terms  corresponding to words with length less
than or equal to $j$ and that they satisfy 
equations
among~(\ref{eq:primera_ec_congruencias})--(\ref{eq:ultima_ec_congruencias}) 
that involve words of length $j$ or shorter. Then we shall define
terms corresponding to words with length  equal to $j+1$  such that
the totality of terms  
defined satisfy  equations
among~(\ref{eq:primera_ec_congruencias})--(\ref{eq:ultima_ec_congruencias}) which
involve words of length $j+1$ or shorter. We have two cases:

\noindent \textbf{Case 1: $j$ odd.} Take
$\alpha$, with $\largo{\alpha}=j$. We have $L_\alpha$ and $R_\alpha$
and by~(\ref{eq:7}), they satisfy $(L_{\alpha },R_{\alpha })
\in \delta^N_{N-j+1} = (\th \circ^N
\epsilon^N_{N-j})\cap (\ths \circ^N \epsilon^N_{N-j})$. We  define  $L_{\alpha i}$ and
$R_{\alpha i}$ for $i=1,\dots,N$ such that:
\begin{equation}\label{eq:23}
\begin{split}
L_\alpha \; \th \; L_{\alpha 1} \;\epsilon_{N-j}^N\; R_{\alpha 1} \;
\th \; L_{\alpha 2} 
\cdots R_{\alpha (k-1)} \;\th \;L_{\alpha k} \;\epsilon_{N-j}^N \;
R_{\alpha k} \;\th \; R_\alpha  \\ 
L_\alpha  \; \ths \;
L_{\alpha (k+1)} \;\epsilon_{N-j}^N\; R_{\alpha (k+1)} \; \ths \;
L_{\alpha (k+2)} \cdots L_{\alpha N} \;\epsilon_{N-j}^N\;
R_{\alpha N} \;\ths  \; R_\alpha.
\end{split}
\end{equation}
The equations
among~(\ref{eq:primera_ec_congruencias})--(\ref{eq:ultima_ec_congruencias})
which involve terms $L_{\mu}$ and $R_{\mu}$ with $\largo{\mu}=j+1$
are~(\ref{eq:25})--(\ref{eq:28}). 
All of them can be inferred from~(\ref{eq:23}).

\noindent \textbf{Case 2: $j$ even.} Take 
$\alpha$, with $\largo{\alpha}=j$. We define  $L_{\alpha i}$ and
$R_{\alpha i}$ for $i=1,\dots,N$. By~(\ref{eq:24}) and by the
definition of $\epsilon_{N-j+1}^N$ we may define our terms satisfying:
\begin{equation*}
\begin{split}
L_\alpha\; \phi \; L_{\alpha 1} \;\delta_{N-j}^N\; R_{\alpha 1} \;
\phi \; L_{\alpha 2} 
\cdots R_{\alpha (k-1)} \;\phi \;  L_{\alpha k} \;\delta_{N-j}^N \; R_{\alpha k} \;\phi \; R_\alpha \\
L_\alpha \;\phi\; L_{\alpha (k+1)} \;\delta_{N-j}^N\; R_{\alpha (k+1)} \; \phis \;
L_{\alpha (k+2)} \cdots  
R_{\alpha N} \;\phis  R_\alpha.
\end{split}
\end{equation*}

From this we immediately conclude~(\ref{eq:8})--(\ref{eq:ultima_ec_congruencias}).
\medskip
Let  $V\subseteq
Z^\infty$ be a finite set of variables   such that if we replace $\th$, $\ths$,
$\phi$ and $\phis$, respectively, by the following compact congruences: 
\begin{align*}
\th_0 &:= \Cg(\vec 0,\vec z) \o \Cg(x,y) \o \bigvee\{\Cg(p,x_{p,q}):x_{p,q}  \in V\}
 \\
\ths_0 &:= \Cg(\vec 1,\vec z) \o \bigvee\{\Cg(x_{p,q},q): x_{p,q} \in V\} \\
\phi_0 &:= \Cg(\vec 0,\vec z)  \o \bigvee\{\Cg(p,y_{p,q}):y_{ p,q} \in V\}\\
\phis_0 &:= \Cg(\vec 1,\vec z) \o \bigvee\{\Cg(y_{p,q},q): y_{p,q} \in V\}
\end{align*}
we still obtain congruential relations~(\ref{eq:primera_ec_congruencias})--(\ref{eq:ultima_ec_congruencias}) 
(excepting (\ref{eq:29}), (\ref{eq:7}) and (\ref{eq:24})). It is clear
that if we enlarge the set $V$ to a new set $X$, the properties enumerated will
still hold.

Let $V_0$ be the union  of $V$ and the (finite) set of variables occurring in  terms
$L_\alpha, R_\alpha$ with $\alpha$ a word. Define:
\[V_{n+1} := V_n \cup \bigcup\{Var(p),Var(q): x_{p,q}
\in V_n \text{ or } y_{p,q} \in V_n\}\]
Hence, for some $M$ we have $V_M =
V_{M+1}$; set 
\[X := \bigl(V_M \cup \{x_{p,q}: y_{p,q} \in V_M\} \cup
\{y_{p,q}: x_{p,q} \in V_M\}\bigl) \setminus \{x,y,z_1,\dots,z_l\}.\]
Order $X$ totally so that $ind:X
\func \omega$ is non decreasing and $x_{p,q}$ is the immediate
predecessor of $y_{p,q}$, and add $x,y,z_1,\dots,z_l$ at the beginning; we have 
\[\vec X =
(x,y,\vec z,x_{s_1,t_1},y_{s_1,t_1},\dots,x_{s_{n},t_{n}},y_{s_{n},t_{n}}) =
(x,y,\vec z,x_1,y_1,\dots, x_n,y_n).\]
We may consider then
$L_\alpha=L_\alpha(\vec X)$ and the same for
$R_\alpha$, and by the remarks after the definition of $Z^\infty$,
we may assume $s_i = s_i(x,y,\vec z,x_1,y_1,\dots,x_{i-1},y_{i-1})$ and the
same for $t_i$. Finally, define $\sigma ,\rho, \sigma^* ,\rho^*$ on
the term algebra  $T(\vec X)$ with respect to $s_i, t_i$.

We claim that these $L$'s, $R$'s, $s$'s and $t$'s satisfy the
Mal'cev condition. Let's check it for
identity~(\ref{eq:dfc_par_phi_1}). Take $\alpha$ with
$0<\largo{\alpha}< N$ an even integer.
By Lemma~\ref{l:igualdad_congr} we have 
\begin{equation*}
\phi = \Cg(\vec X, \rho(\vec X)). 
\end{equation*}
Since we have
$L_\alpha  \;\phi\; L_{\alpha 1}$ by
equation~(\ref{eq:8}), Lemma~\ref{malsev} gives us terms $p_i$ such that for some tuple
$\vec u$,  $F$ satisfies:
\begin{align*}
L_\alpha &= p_1\bigl(\vec X, \vec u\bigr) \\
p_1\bigl(\rho(\vec X), \vec u\bigr) &= p_2 \bigl(\rho(\vec X), \vec u\bigr) \\
p_2\bigl(\vec X, \vec u\bigr) &= p_3 \bigl(\vec X, \vec u\bigr) \\ & \cdots \\
p_m\bigl(\rho(\vec X), \vec u\bigr) &= L_{\alpha 1}
\end{align*}
Since $\vec u = \vec u (\vec X, \vec Y)$ can be construed as members of $T( Z^\infty)$, we obtain the following laws
for $\V$:
\begin{align*}
 L_\alpha(\vec X) &\id p_1(\vec X,  \vec u (\vec X, \vec Y))  \\
 p_1(\rho(\vec X),  \vec u (\vec X, \vec Y)) &\id p_2 (\rho(\vec X),  \vec u (\vec X, \vec Y)) \\
 p_2(\vec X,  \vec u (\vec X, \vec Y)) &\id p_3 (\vec X,  \vec u (\vec X, \vec Y)) 	\\ 
		& \cdots \\
 p_m(\rho(\vec X),  \vec u (\vec X, \vec Y)) &\id L_{\alpha 1}(\vec X)
\end{align*}
Replacing $\vec X$ by $\rho(\vec X)$ everywhere and noting that
$\rho(\rho(\vec X))=\rho(\vec X)$, we have
\begin{align*}
 L_\alpha(\rho(\vec X)) &\id p_1(\rho(\vec X), \vec u (\rho(\vec X), \vec Y))  \\
  p_1(\rho(\vec X), \vec u (\rho(\vec X), \vec Y)) &\id p_2 (\rho(\vec X), \vec u (\rho(\vec X), \vec Y)) \\
 p_2(\rho(\vec X), \vec u (\rho(\vec X), \vec Y)) &\id p_3 (\rho(\vec X), \vec u (\rho(\vec X), \vec Y)) 	\\ 
 & \cdots \\
  p_m(\rho(\vec X), \vec u (\rho(\vec X), \vec Y))  &\id L_{\alpha 1}(\rho(\vec X))
\end{align*}
and by transitivity,
\[ \V \models L_{\alpha} (\rho(\vec X)) \id L_{\alpha 1}(\rho(\vec X)),\]
which is what we were looking for. The other identities are obtained similarly.
\end{proof} 

The proof of the previous theorem follows the line of a proof for a
Mal'cev condition for BFC. One such condition that closely parallels
ours was personally communicated to us by R.~Willard.

\section{A Canonical Form of DFC}\label{sec:canonical-form-dfc}
We will assume in this section that $\V$ has the Determining
Property. Since the Determining Property implies Left Determining Property, we may define the following formulas
in the language of $\V$:
 \[ \Psi_m := \bigwedge_{\largo{\alpha} = m} 
\left(\Bigl(\bigwedge_{\gamma\neq\varepsilon}
L_{\alpha\gamma}(\termuno) = R_{\alpha\gamma}(\termuno)\Bigr)
\ \impl \ L_\alpha(\termuno) = R_\alpha(\termuno)\right) \]
where every subindex varies over words of length less than or
equal to $N$; so an expression of the form
``$\bigwedge_{\gamma\neq\varepsilon} L_{\alpha\gamma} = R_{\alpha\gamma}$'' should be
read like ``$\bigwedge \{ L_{\alpha\gamma} = R_{\alpha\gamma} :\gamma\neq\varepsilon
\text{ and } \largo{\alpha\gamma}\leq N \}$''. Hence $\Psi_N=\bigl(\bigwedge_{\largo{\beta}=N} L_\beta(\termuno)=R_\beta(\termuno)\bigr)$. (The
antecedent ``vanishes''.) 
\begin{lemma}\label{l:def_q_centro}
Let $A\in\V$  and let $\phi, \phis\in \CON A$ such that $\phi
\romb \phis =\0$, and $\vec 0 \,\phi\,\vec e\,\phis\,\vec 1$. Then for
all $c,d\in A$,  $A$ satisfies $ \Phi_1(c,d,\vec e)$, where
\begin{equation}\label{eq:def_q_centro}
\Phi_1(x,y,\vec z)\ :=\ 
\exists y_1 \forall x_1  \dots  \exists y_n \forall x_n \bigwedge_{m = 1}^{k}
\Psi_{2m}
\end{equation}
with $n,k$ as in Theorem~\ref{th:malcev_dfc}.
\end{lemma}
\begin{proof}
Take $b_1$ such that
\[s_1(c,d,\vec e)\stackrel{\phi}{\equiv} \  b_1 \stackrel{\phis}{\equiv}
t_1(c,d,\vec e).\]
Assuming $b_i$ is
already chosen and $a_{i}$ is given, define $b_{i+1}$ such that
\begin{equation*}\label{eq:selecc_ei_Phi1}
s_{i+1}(c,d,\vec e,a_1,b_1,\dots,a_{i},b_{i})\stackrel{\phi}{\equiv}  \ b_{i+1}
 \stackrel{\;\phis}{\equiv}t_{i+1}(c,d,\vec e,a_1,b_1,\dots,a_{i},b_{i}). 
\end{equation*}

The construction of $b_i$'s then corresponds to
equations in Corollary~\ref{l:recursion_bes}. Hence, 
 (\ref{eq:5}) imply that $A$ satisfies
\begin{equation*}
\begin{split}
L_\alpha(c,d,\vec e,a_1,b_1,\dots,a_{n},b_{n}) &\stackrel{\phi}{\equiv}
L_\alpha(\rho(c,d,\vec e,a_1,b_1,\dots,a_{n},b_{n}))\\
R_\alpha(c,d,\vec e,a_1,b_1,\dots,a_{n},b_{n}) &\stackrel{\phi}{\equiv}
R_\alpha(\rho(c,d,\vec e,a_1,b_1,\dots,a_{n},b_{n}))\\
L_\alpha(c,d,\vec e,a_1,b_1,\dots,a_{n},b_{n}) &\stackrel{\;\phis}{\equiv}
L_\alpha(\rho^*(c,d,\vec e,a_1,b_1,\dots,a_{n},b_{n})) \\
R_\alpha(c,d,\vec e,a_1,b_1,\dots,a_{n},b_{n}) &\stackrel{\;\phis}{\equiv}
R_\alpha(\rho^*(c,d,\vec e,a_1,b_1,\dots,a_{n},b_{n})).
\end{split}
\end{equation*}
for all $\alpha$. These together with equations~(\ref{eq:dfc_N}) imply
that for each $\beta$ with $\largo{\beta}=N$, 
\begin{equation}\label{eq:21}
\begin{split}
L_\beta(c,d,\vec e,a_1,b_1,\dots,a_{n},b_{n}) &\stackrel{\phi}{\equiv}
R_\beta(c,d,\vec e,a_1,b_1,\dots,a_{n},b_{n}) \\
L_\beta(c,d,\vec e,a_1,b_1,\dots,a_{n},b_{n}) &\stackrel{\;\phis}{\equiv} 
R_\beta(c,d,\vec e,a_1,b_1,\dots,a_{n},b_{n}).
\end{split}
\end{equation}
Since $\phi \cap \phis = \0$, this yields $A\models \Psi_N(c,d,\vec e,a_1,b_1,\dots,a_{n},b_{n})$.  

Take nonempty $\alpha$ with  $\largo{\alpha}<N$ even and suppose 
\[A\models
\bigwedge_{\gamma\neq\varepsilon} 
L_{\alpha\gamma}(c,d,\vec e,a_1,b_1,\dots,a_{n},b_{n}) =
R_{\alpha\gamma}(c,d,\vec e,a_1,b_1,\dots,a_{n},b_{n}).\] 
We can see that 
$L_\alpha(c,d,\vec e,a_1,b_1,\dots,a_{n},b_{n})
\stackrel{\phi}{\equiv}  R_\alpha(c,d,\vec
e,a_1,b_1,\dots,a_{n},b_{n})$, by using
equations~(\ref{eq:dfc_par_phi_1}), (\ref{eq:dfc_par_phi})
and~(\ref{eq:dfc_par_phi_k}) as follows:
\begin{align*}
L_\alpha(c,d,\vec e,a_1,b_1,\dots,a_{n},b_{n}) &\stackrel{\phi}{\equiv}
L_\alpha(\rho(c,d,\vec e,a_1,b_1,\dots,a_{n},b_{n})) 
&& \text{by equation~(\ref{eq:5})}\\
&= L_{\alpha 1}(\rho(c,d,\vec e,a_1,b_1,\dots,a_{n},b_{n})) 
&& \text{by identity~(\ref{eq:dfc_par_phi_1})} \\
&\stackrel{\phi}{\equiv} L_{\alpha 1}(c,d,\vec e,a_1,b_1,\dots,a_{n},b_{n}) 
&& \text{by equation~(\ref{eq:5})}\\
& = R_{\alpha 1}(c,d,\vec e,a_1,b_1,\dots,a_{n},b_{n})
&& \text{by hypothesis} \\
& \stackrel{\phi}{\equiv} R_{\alpha 1}(\rho(c,d,\vec e,a_1,b_1,\dots,a_{n},b_{n})) 
&& \text{by equation~(\ref{eq:5})}\\
& = \ \cdots && \text{using equations~(\ref{eq:dfc_par_phi}),
  (\ref{eq:21})} \\
& = \ \cdots && \text{and iterating\dots}\\
& = R_{\alpha k}(\rho(c,d,\vec e,a_1,b_1,\dots,a_{n},b_{n})) \\
& = R_{\alpha}(\rho(c,d,\vec e,a_1,b_1,\dots,a_{n},b_{n})) && \text{using
identity~(\ref{eq:dfc_par_phi_k})}\\
& \stackrel{\phi}{\equiv} R_{\alpha}(c,d,\vec e,a_1,b_1,\dots,a_{n},b_{n}) &&  \text{by equation~(\ref{eq:5})} 
\end{align*} 

It can be proved in an entirely analogous fashion (by using
equations~(\ref{eq:dfc_par_phi*})) that $L_\alpha(c,d,\vec e,a_1,b_1,\dots,a_{n},b_{n})
\stackrel{\;\phis}{\equiv}
R_\alpha(c,d,\vec e,a_1,b_1,\dots,a_{n},b_{n})$, which yields 
\[A \models   L_\alpha(c,d,\vec e,a_1,b_1,\dots,a_{n},b_{n}) =
R_\alpha(c,d,\vec e,a_1,b_1,\dots,a_{n},b_{n}),\] 
and we have proved the lemma. 
\end{proof}

\begin{theorem}\label{th:def_kernel_dfc}
Let $\V$ be a variety with the Determining Property, let $A\in\V$  and let $\th, \ths\in \CON A$ such that $\th
\romb \ths=\0$, and $\vec 0 \,\th\,\vec e\,\ths\,\vec 1$. Then 
$c\stackrel{\th}{\equiv} d$ if and only if $A\models
\Phi_1(c,d,\vec e) \y \Phi_2(c,d,\vec e)$ where $\Phi_1$ and $\Phi_2$
are defined as follows:
\begin{equation*}\label{eq:def_kernel_dfc}
\begin{split}
\Phi_1(x,y,\vec z)\ &:=\ 
\exists y_1 \forall x_1  \dots  \exists y_n \forall x_n \bigwedge_{m = 1}^{k}
\Psi_{2m}\\
\Phi_2(x,y,\vec z)\ &:=\ \exists x_1 \forall y_1 \dots \exists x_n \forall y_n  \bigwedge_{m = 1}^{k}
\Psi_{2m-1}  
\end{split}
\end{equation*}
\end{theorem}

\begin{proof}
Since the definition of  $\Phi_1$ is the same as the one given by
formula~(\ref{eq:def_q_centro}) in Lemma~\ref{l:def_q_centro}, we only have to worry about $\Phi_2(c,d,\vec e)$.
\smallskip

($\ent$) Assume $c\stackrel{\th}{\equiv} d$. Much in the same
way as in the proof
of Lemma~\ref{l:def_q_centro}, define $a_1$ such that 
\[s_1(c,d,\vec e)\stackrel{\th}{\equiv} \  a_1 \stackrel{\ths}{\equiv}
t_1(c,d,\vec e);\]
and assuming $a_i$ is
already chosen and $b_i$ is given, let
\begin{equation*}
s_{i+1}(c,d,\vec e,a_1,b_1,\dots,a_{i},b_{i})\stackrel{\th}{\equiv}  \ a_{i+1}
 \stackrel{\;\ths}{\equiv}t_{i+1}(c,d,\vec e,a_1,b_1,\dots,a_{i},b_{i}). 
\end{equation*}
This choice conforms the pattern of
Corollary~\ref{l:recursion_aes}, so we obtain
\begin{equation}
\begin{split}\label{eq:18}
L_\alpha(c,d,\vec e,a_1,b_1,\dots,a_{n},b_{n}) &\stackrel{\th}{\equiv}
L_\alpha(\sigma(c,d,\vec e,a_1,b_1,\dots,a_{n},b_{n}))\\
R_\alpha(c,d,\vec e,a_1,b_1,\dots,a_{n},b_{n}) &\stackrel{\th}{\equiv}
R_\alpha(\sigma(c,d,\vec e,a_1,b_1,\dots,a_{n},b_{n}))
\end{split}
\end{equation}
\begin{equation}
\begin{split}\label{eq:19}
L_\alpha(c,d,\vec e,a_1,b_1,\dots,a_{n},b_{n}) &\stackrel{\;\ths}{\equiv}
L_\alpha(\sigma^*(c,d,\vec e,a_1,b_1,\dots,a_{n},b_{n})) \\
R_\alpha(c,d,\vec e,a_1,b_1,\dots,a_{n},b_{n}) &\stackrel{\;\ths}{\equiv}
R_\alpha(\sigma^*(c,d,\vec e,a_1,b_1,\dots,a_{n},b_{n})).
\end{split}
\end{equation}
by equations~(\ref{eq:4}). 

If we suppose now that 
\[A\models \bigwedge_{\gamma\neq\varepsilon}
L_{\alpha\gamma}(c,d,\vec e,a_1,b_1,\dots,a_{n},b_{n}) =
R_{\alpha\gamma}(c,d,\vec e,a_1,b_1,\dots,a_{n},b_{n})\]
 for some  $\alpha$ with  $\largo{\alpha}<N$ odd, we'll be able to prove
$L_{\alpha}(c,d,\vec e,a_1,b_1,\dots,a_{n},b_{n}) = R_{\alpha}(c,d,\vec e,a_1,b_1,\dots,a_{n},b_{n})$ by showing (in
the same way as in Lemma~\ref{l:def_q_centro}) that:
\begin{itemize}
\item $L_{\alpha}(c,d,\vec e,a_1,b_1,\dots,a_{n},b_{n})\stackrel{\th}{\equiv}
R_{\alpha}(c,d,\vec e,a_1,b_1,\dots,a_{n},b_{n})$ (this can be accomplished  using ~(\ref{eq:2})
and~(\ref{eq:18})), and
\item $L_{\alpha}(c,d,\vec e,a_1,b_1,\dots,a_{n},b_{n})
\stackrel{\;\ths}{\equiv} R_{\alpha}(c,d,\vec e,a_1,b_1,\dots,a_{n},b_{n})$ (by~(\ref{eq:3})
and~(\ref{eq:19})). 
\end{itemize}
\smallskip

($\tne$) Assume $A\models \Phi_2(c,d,\vec e)$. Take $b_1$ such that 
\[s_1(c,d,\vec e)\stackrel{\th}{\equiv} \  b_1 \stackrel{\ths}{\equiv}
t_1(c,d,\vec e).\]
Let $a_1$
be given by the outermost existential quantifier of $\Phi_2$.

Assuming $b_{i}$ is
already chosen and $a_i$ is the corresponding witness for $\Phi_2$, let
\begin{equation*}\label{eq:22}
s_{i+1}(c,d,\vec e,a_1,b_1,\dots,a_{i},b_{i})\stackrel{\th}{\equiv}  \ b_{i+1}
 \stackrel{\;\ths}{\equiv}t_{i+1}(c,d,\vec e,a_1,b_1,\dots,a_{i},b_{i}). 
\end{equation*}
This selection conforms scheme of Corollary~\ref{l:recursion_bes} (with
$\phi:=\th$ and $\phis:=\ths$) and satisfies the matrix of $\Phi_1$, as it
was seen in the  proof of Lemma~\ref{l:def_q_centro}. Hence we have, respectively
\begin{equation}\label{eq:20}
\begin{split}
L_\alpha(c,d,\vec e,a_1,b_1,\dots,a_{n},b_{n}) &\stackrel{\th}{\equiv}
L_\alpha(\rho(c,d,\vec e,a_1,b_1,\dots,a_{n},b_{n}))\\
R_\alpha(c,d,\vec e,a_1,b_1,\dots,a_{n},b_{n}) &\stackrel{\ths}{\equiv}
R_\alpha(\rho(c,d,\vec e,a_1,b_1,\dots,a_{n},b_{n}))
\end{split}
\end{equation}
and
\begin{equation*}
A\models \Bigl(\bigwedge_{m=1}^{N} \Psi_m\Bigr) (c,d,\vec
e,a_1,b_1,\dots,a_n,b_n).
\end{equation*}
From an easy inspection of the form of $\Psi_m$, it can be deduced
that
\begin{equation}\label{eq:long_1}
A\models \bigwedge_{j=1}^{N} L_j(c,d,\vec e,a_1,b_1,\dots,a_{n},b_{n}) = R_j(c,d,\vec e,a_1,b_1,\dots,a_{n},b_{n}),  
\end{equation}
Therefore,
\begin{align*}
c & =L_\varepsilon(c,d,\vec e,a_1,b_1,\dots,a_{n},b_{n}) && \text{by
  identities~(\ref{eq:1})} \\
&\stackrel{\th}{\equiv} L_\varepsilon(\rho(c,d,\vec e,a_1,b_1,\dots,a_{n},b_{n})) && \text{by
  equations~(\ref{eq:20})}\\ 
& = L_1(\rho(c,d,\vec e,a_1,b_1,\dots,a_{n},b_{n})) &&
\text{by identity~(\ref{eq:dfc_par_phi_1}), with $\alpha=\varepsilon$} \\
&\stackrel{\th}{\equiv} L_1(c,d,\vec e,a_1,b_1,\dots,a_{n},b_{n}) && \text{by  equations~(\ref{eq:20})}\\ 
 & = R_1(c,d,\vec e,a_1,b_1,\dots,a_{n},b_{n}) && \text{by (\ref{eq:long_1})}  \\
&\stackrel{\th}{\equiv}  R_1(\rho(c,d,\vec e,a_1,b_1,\dots,a_{n},b_{n})) && \text{by
  equations~(\ref{eq:20})}\\  
&= L_2(\rho(c,d,\vec e,a_1,b_1,\dots,a_{n},b_{n})) && \text{by identities~(\ref{eq:dfc_par_phi})} \\
&\stackrel{\th}{\equiv} \ \cdots && \text{using
  equations~(\ref{eq:dfc_par_phi}), (\ref{eq:long_1})} \\
&\stackrel{\th}{\equiv} \ \cdots && \text{and iterating\dots}\\
& = R_N(\rho(c,d,\vec e,a_1,b_1,\dots,a_{n},b_{n})) && \text{and using
equation~(\ref{eq:dfc_par_phi_k}):}\\
& = R_{\varepsilon}(\rho(c,d,\vec e,a_1,b_1,\dots,a_{n},b_{n})) \\
&\stackrel{\th}{\equiv} R_\varepsilon(c,d,\vec e,a_1,b_1,\dots,a_{n},b_{n}) && \text{by
  equations~(\ref{eq:20})}\\ 
& = d && \text{by identities~(\ref{eq:1})}
\end{align*} 
Hence $c \stackrel{\th}{\equiv} d$.
\end{proof}

\section{Central Elements in a Variety with DFC}\label{sec:centr-elem-vari}
In the Appendix
,  we will prove a
preservation result (Theorem~\ref{t:preserv_ker}) that implies that formula
$\Phi_1 \y \Phi_2$ of Theorem~\ref{th:def_kernel_dfc} is preserved by
taking direct 
factors and direct products  (taking $\tau_\alpha$ to be
``$L_\alpha(\vec X)=R_\alpha(\vec X)$''). Call $\Phi$ the conjunction $\Phi_1 \y \Phi_2$.
\begin{lemma}\label{l:preserv_Z}
Assume $\V$ has the Determining Property. Then there is a set $\Sigma$ of first order formulas
such that for every $A\in\V$, $\vec e,\vec f \in A^l$ we have
that $\vec e$ and $\vec f$ are complementary 
central elements if and only if $A\models
\zeta(\vec e,\vec f)$ for every $\zeta\in\Sigma$. Moreover, each formula in
$\Sigma$ is preserved by taking direct factors.
\end{lemma}
\begin{proof}
The
following formulas in the language of $\V$ will assert the properties
needed to force $\Phi(\cdot,\cdot,\vec e)$ and  $\Phi(\cdot,\cdot,\vec f)$ to
define the pair of complementary factor congruences associated with
$\vec e$ and $\vec f$.

\begin{itemize}
\item $CAN(\vec e,\vec f)= \bigwedge_{i=1}^l \Phi(0_i,e_i,\vec e) \y \bigwedge_{i=1}^l \Phi(1_i,f_i,\vec e)$

This sentence says that  $\vec e$ it is related to $\vec 0$ and $\vec f$ to
$\vec 1$ via $\Phi(\cdot,\cdot,\vec e)$.
\item $PROD(\vec e,\vec f)=\forall x,y\exists z\ 
\Bigl(\Phi(x,z,\vec e)\wedge \Phi(z,y,\vec f)\Bigr)$

The relational product of  $\Phi(\cdot,\cdot,\vec e)$ and
$\Phi(\cdot,\cdot,\vec f)$ is the universal congruence.
\item $INT(\vec e,\vec f)=\forall x,y\ 
\Bigl(\Phi(x,y,\vec e)\wedge \Phi(x,y,\vec f)\rightarrow x=y\Bigr)$

Their intersection is $\Delta$.
\item$REF(\vec e,\vec f)=\forall x\ \Phi(x,x,\vec e)$

$\Phi(\cdot,\cdot,\vec e)$ is reflexive.
\item $SYM(\vec e,\vec f)=\forall x,y,z\ \Bigl(\Phi(x,y, \vec e)\wedge \Phi(y,z,\vec e)\wedge 
\Phi(z,x,\vec f) \rightarrow z=x\Bigr)$
\item $TRANS(\vec e,\vec f)=\forall x,y,z,u
\;\Bigl(\Phi(x,y,\vec e)\wedge 
\Phi(y,z,\vec e)\wedge \Phi(x,u,\vec e)\wedge \Phi(u,z,\vec f)\rightarrow u=z\Bigr)$

The reader may verify that these two sentences (in conjunction with
the previous ones) say that  $\Phi(\cdot,\cdot,\vec e)$ is symmetric and transitive.

\item For each $m$-ary function symbol $F,$ define:
\begin{multline*}
PRES_{F}(\vec e,\vec f)=\forall u_{1},v_{1},\dots,u_{m},v_{m} \\ 
\Bigl(\bigwedge_{j}\Phi(u_{j},v_{j},\vec e)\Bigr)\;\wedge 
\Phi(F(u_{1},\dots,u_{m}),z,\vec e)\wedge 
\Phi(z,F(v_{1},\dots,v_{m}),\vec f)\rightarrow\\ 
\rightarrow z=F(v_{1},\dots,v_{m})
\end{multline*}
These sentences ensure  $\Phi(\cdot,\cdot,\vec e)$ is preserved by the
basic operations of $\V$.
\end{itemize}
\medskip

Finally, define $CAN\qua$, $REF\qua,$ $SYM\qua,$ $TRANS\qua$ and $PRES_{F}\qua$ to 
be
the result of interchanging $\vec e$ with $\vec f$ in $CAN,$ $REF,$ $SYM,$ 
$TRANS$
and $PRES_{F},$ respectively, and let $\Sigma $ be the union of the
following two sets
\[\{CAN, PROD,INT,
REF,SYM,TRANS,CAN\qua,REF\qua,SYM\qua,TRANS\qua\},\]
\[\{PRES_{F},PRES_{F}\qua : F\text{ a function symbol}\}.\]

Now it is immediate to check that  $\vec e$ and $\vec f$ are complementary
central elements if they satisfy
all formulas in $\Sigma$. To see the converse, note that  if $\vec e$ and
$\vec f$ are complementary central elements,
there is an isomorphism $A\func A_0\times A_1$ such that $\vec e, \vec
f$ correspond to $[\vec 0,\vec 1], [\vec 1,\vec 0]$, respectively, and
Theorem~\ref{th:def_kernel_dfc} guarantees 
that $\Sigma$ will hold.

To see that $\Sigma$ is preserved by direct factors, we first note
that each one of  $CAN$, $CAN'$
$PROD$, $PROD\qua$, $REF$ and $REF\qua$ is obtained by  quantification of a formula preserved
by taking direct factors by Theorem~\ref{t:preserv_ker}. In 
second place,  the remainder of axioms in $\Sigma $ are of the form 
$\forall \vec{x}\;(\tau (\vec e,\vec f,\vec{x})\rightarrow x_{i}=x_{j})$ 
with $REF(\vec e,\vec f) \y REF\qua(\vec e,\vec f) \rightarrow \exists \vec{x}\;\tau 
(\vec e,\vec f,\vec{x})$ universally valid, and since $\forall \vec{x}\;(\tau 
(\vec e,\vec f,\vec{x})\rightarrow x_{i}=x_{j})\wedge \exists
\vec{x}\;\tau (\vec e,\vec f,\vec{x})$
is preserved by taking direct factors (whenever $\tau (\vec e,\vec f,\vec{x})$
is preserved by taking direct factors and direct products,) we have the result. 
\end{proof}
\begin{corollary}\label{c:preserv_Z}
Assume $\V$ has the Determining Property. Then if $[\vec e_0, \vec e_1]$ is a central
element of $A_0\times  A_1$, then $\vec e_i$ is a central element of
$A_i$, $i=0,1$. 
\end{corollary}
\begin{proof}
Immediate by the previous lemma.
\end{proof}
We use $\V_{DI}$ to denote the class of directly indecomposable
members of $\V$. 
\begin{corollary}\label{co:def_DI}
If $\V$ is a variety over a finite language and it has the Determining Property then
$\V_{DI}$ is axiomatizable by a set of first order sentences.
\end{corollary}
\begin{proof}
The set $\Sigma= \Sigma(\vec e,\vec f)$ in Lemma~\ref{l:preserv_Z} is finite if
the language is finite. Hence
\[\vec 0 \neq \vec 1 \ \y \ \forall \vec e, \vec f \ \; \bigwedge
\Sigma(\vec e,\vec f) \ \impl\; \bigl((\vec e =
\vec 0 \y \vec f= \vec 1) \o
(\vec e = \vec 1 \y \vec f=\vec 0)\bigr)\]
together with axioms for $\V$ defines $\V_{DI}$.
\end{proof}

\begin{lemma}\label{l:DPimplBFC}
The Determining Property implies BFC.
\end{lemma}
\begin{proof}
By Bigelow and Burris~\cite{1}, we only need to check that if $A =
A_0\times A_1$, and $\th$ is a factor congruence on $A$, then 
\[\{\<(a,b),(c,b)\>: b\in A_1 \text{ and } \exists a',c' \<a,a'\>
\,\th\, \<c,c'\>\} \subseteq \th.\]
Let $\vec e=[\vec e_0, \vec e_1]$ be the central element associated to $\th$, so
``$x\,\th\,y$'' is 
defined by $\Phi(x,y,\vec e)$. We have 
\[ \<a,a'\> \,\th\, \<c,c'\>  \text{ iff }A_0\times A_1 \models \Phi
(\<a,a'\> , \<c,c'\>,[\vec e_0, \vec e_1]). \]
By Theorem~\ref{t:preserv_ker}, this implies 
\[A_0 \models \Phi(a,c, \vec e_0).\]
Now Corollary~\ref{c:preserv_Z} ensures $\vec e_1$ is central in $A_1$, and
 hence $A_1 \models \Phi(b,b, \vec e_1)$. Since $\Phi$ is
 preserved by direct products, we obtain
\[A_0\times A_1 \models \Phi
(\<a,b\> , \<c,b\>,[\vec e_0, \vec e_1]),\]
and then $\<a,b\>\,\th\, \<c,b \>$. 
\end{proof}

\section{The Main Theorem}\label{sec:main-theorem}

\begin{proof}[Proof of Theorem~\ref{th:principal}]
(5)$\ent$(2) Suppose we have a pair
of complementary factor congruences $\phi$ and $\phis$ such that  $
\vec 0
\,\phi\, \vec e \,\phis \,\vec 1$. Suppose now that also
$\th\romb\ths =\0$ and $ \vec 0
\,\th\, \vec e \,\ths \,\vec  1$. Then $\vec 0 \stackrel{\phi
  \circ \ths }{\equiv}\vec 1$ and hence
$\phi \vee \ths =\nabla$. So we have
\[(\phi \vee \ths ) \cap \th = \th\]
By BFC, we obtain $\phi\cap\th = \th$ and then $\phi \subseteq
\th$. By symmetry, we obtain $\phi = \th$ and $\phis = \ths$.

(2)$\ent$(3) Theorem~\ref{th:def_kernel_dfc}.

(3)$\ent$(4) Obvious.

(4)$\ent$(1) Immediate.

(2)$\ent$(5) Lemma~\ref{l:DPimplBFC}.

(1)$\ent$(5) Define  $\til 0_{i}$ and $\til 1_{i}$ with
   $i=1,\dots, 2l$ in the following way:

\[(\til 0_{1},\dots, \til 0_{2l}) := (0_{1},\dots 0_{l},1_{1},\dots 1_{l}),\]
\[(\til 1_{1},\dots, \til 1_{2l}) := (1_{1},\dots 1_{l},0_{1},\dots,
   0_{l}).\]

It can be easily checked (using the Weak Determining Property) that with these $\til 0$'s and
$\til 1$'s we have the Determining Property. Since (2)$\ent$(5), we have our result.
\end{proof}

\section{Examples}\label{sec:examples}

\subsection{$\vec 0$ \& $\vec 1$ does not imply BFC}\label{sec:01-does-not}
 The variety $\V_L$ with language $\{+,*,0,1\}$ given by the following
 set of equations $\Sigma$: 
\begin{align*}
x + 0 &= x \\
x + 1 &= x*1 \\
x * 0 &= 0
\end{align*}
 has $\vec 0$ \& $\vec 1$. Next we will define various algebras in $\V_L$. In the
 first place take $L_\omega := \langle \omega,\maso,\poro,0,1 \rangle$,
where
\begin{align*}
0 \maso 1 &:= 0		&	0 \poro 1 &:= 0 \\
1 \maso 1 &:= 1		&       1 \poro 1 &:= 1\\
x \maso 0 &:= x		&	x \poro 0 &:= 0
\end{align*}
for all $x\in\omega$ and  
\[z \maso y := 2,\qquad z \poro y := 2\]
for all $z,y\in\omega$ not previously considered. For each $n\geq 2$, $L_n$ will denote  the subalgebra  of
$L_\omega$ with universe $n=\{0,1,\dots,n-1\}$. 
Now define $D_n$ to be the subalgebra of $L_2 \times L_\omega$ with
universe $(2 \times n) \cup \{(1,n)\}$.

Define the following subsets of $2 \times \omega$:
\begin{equation*}
\begin{split}
P_0	&:= \{(0,j)  \;|\; 3 \leq j \}\\
P_1	&:=\{(1,j)  \;|\; 3 \leq j \}
\end{split}
\end{equation*}
Then  $2 \times \omega = (2\times 3) \cup P_0 \cup P_1$. Note
that for all $z\in (2 \times \omega) \setminus \{(0,0), (1,0)\}$
and  for all $x,y\in P_1$ we have:
\begin{equation}\label{eq:33}
\begin{split}
x\masd z &= y\masd z \\  
z\masd x &= z\masd y  \\
x\pord z &=y\pord z \\
 z\pord x &=z\pord y. 
\end{split}
\end{equation}
  
\begin{lemma}\label{lem:iso_parc}
Every injective \emph{partial} function $f:D_n \func  (L_2\times L_n)$
which fixes $ (2 \times 3) \cup P_0$ is a partial isomorphism between
$D_n$ and $L_2 \times L_n$. 
\end{lemma}
\begin{proof}
It's straightforward to see (using equations~(\ref{eq:33})) that if $B\subseteq
P_1$ and $\sigma$ is any permutation 
of $P_1$, then $(2 \times 3) \cup P_0  \cup B$ and $(2 \times 3) \cup
P_0 \cup \sigma(B)$ are  subalgebras of $L_2\times L_\omega$
and 
\[\bar \sigma(x) := \begin{cases} 	x
  & x\in (2 \times 3) \cup P_0\\ 
			\sigma(x)	& x \in B, \end{cases}\]
is an isomorphism between them. 

Since $f$ is a restriction of such an isomorphism $\bar\sigma$, it is a
partial isomorphism.
\end{proof}

Recall that $(\V_L)_{DI}$ is the class of directly indecomposable
members of $\V_L$. 
\begin{lemma}\label{l:muchas_una}
Let $\V$ be a variety. If $\V_{DI}$ is axiomatizable by a set of first
order sentences, then it is finitely axiomatizable relative to $\V$.
\end{lemma}
\begin{proof}
First note that an ultraproduct of directly decomposable algebras is
again decomposable. Let $\Sigma$ be a set of first order sentences
axiomatizing $\V_{DI}$. By way of contradiction, suppose that $\V_{DI}$ is
not finitely axiomatizable relative to $\V$. Hence, for each finite
$\Sigma_0\subseteq \Sigma$  there exists $A_{\Sigma_0}\in\V
\setminus \V_{DI}$ satisfying $\Sigma_0$. Now it is easy to construct
an ultraproduct $U$ of these decomposable algebras in such a way that
$U$ satisfies $\Sigma$, an absurdity.
\end{proof}
\begin{theorem}\label{th:no_es_primer_ord}
$(\V_L)_{DI}$ is not axiomatizable by a set of first order sentences.
\end{theorem}
\begin{proof}
We will first prove that player ``\DEFENSE{}'' has a winning strategy
for the back-and-forth (or ``Ehrenfeucht'') game of length $n-3$
played on $D_n$ and $ L_2 \times L_n$. The strategy is as follows:
\begin{itemize}
\item If \ATTACK{} chooses an element in  $(2 \times 3) \cup P_0$ 
(in either algebra), \DEFENSE{}
will choose the same element in the other algebra.
\item If \ATTACK{} chooses an element in $P_1$, \DEFENSE{}
will choose an element in the $P_1$-part of the other
algebra, which has not been 
chosen up to this point.
\end{itemize}
There are $n-3$ elements in $P_1\cap (L_2\times L_n)$, so these
instructions work up to $n-3$ moves. Let's call $g$ the partial
function defined by this game. By
Lemma~\ref{lem:iso_parc}, $g$ is a partial isomorphism and 
we have proved our first claim.

Now suppose $\phi$ is a sentence such that  $(\V_L)_{DI}\models
\phi$. By the above
strategy we have that for every sufficiently large $n$, $D_n \models
\phi$  if and only if $L_2 \times L_n 
\models \phi$. By taking $n$ such that  $2n+1 =$ cardinal of $D_n$ is a prime number, we
obtain $D_n \in (\V_L)_{DI}$. We conclude that there are decomposable
algebras satisfying $\phi$, and hence  $(\V_L)_{DI}$ cannot be defined
by a single first order sentence. Using Lemma~\ref{l:muchas_una} we
have our result.
\end{proof}
\begin{corollary}
$\V_L$ has not DFC.
\end{corollary}
\begin{proof}
Since DFC is equivalent to the Determining Property, we can use Corollary~\ref{co:def_DI}
and Theorem~\ref{th:no_es_primer_ord}. 
\end{proof}

Indication that $(\V_L)_{DI}$ might not be definable was discovered
by using the ``Universal Algebra Calculator'' program, designed by
Ralph Freese and Emil Kiss~\cite{UA}.

\subsection{Semilattice Expansions}\label{sec:semil-expans}

Throughout this section, we will suppose that $\V$ is a  variety with $\vec 0$ \& $\vec 1$ for which there
exists a binary term $\vee $ such that for every $A\in  
\mathcal{V}$, $\vee ^{A}$ is a semilattice operation on $A.$ We will
keep the assumption that the
language of $\V$ has at least one constant.

First, we observe that by Lemma~\ref{malsev} together with the
observation that  $(x,y)\in \nabla ^{F}=\Cg^{F}(\vec 0, \vec 1)$
(where  $F\in \mathcal{V}$ is the free algebra freely generated 
by $\{x,y\}$), we obtain $(2+l)$-ary terms $u_{i}(x,y,\vec{z}),$ $i=1,\dots,k,$ such 
that the following identities hold in $\mathcal{V}$

\begin{equation}\label{eq:34}
  \begin{split}
    x&\id u_{1}(x,y,\vec 0)\\
    u_{i}(x,y,\vec{1})&\id u_{i+1}(x,y,\vec 1) \text{ with $i$ odd}\\
    u_{i}(x,y,\vec 0)&\id u_{i+1}(x,y,\vec 0) \text{ with $i$ even}\\
    u_{k}(x,y,\vec{1})&\id y
    \end{split}
\end{equation}
\begin{prop}
The formula
\[\Phi(x,y,\vec z) := \forall u\;\left( \bigwedge_{i=1}^{k} \Bigl(u_{i}(x,y,\vec
0)\vee
u=u_{i}(x,y,\vec z)\vee u \Bigr) \ \ \longrightarrow \ \ \bigl(x\vee u=y\vee u\bigr)\right)\]
satisfies (3) of Theorem~\ref{th:principal} for $\V$. 
\end{prop}
\begin{proof}
 Let $A, B\in \V$, and $a\in A$,
$b,d\in B$. First we prove that 
\[A\times B \models \Phi\bigl(\<a,b\>, \<a,d\>, [\vec 0, \vec
  1]\bigr)\]
Suppose that for some $(u, v)$ we have
\[A\times B \models \bigwedge_{i=1}^{k}u_{i}(\<a,b\>, \<a,d\>, [\vec 0, \vec 0])\vee
(u,v)=u_{i}(\<a,b\>, \<a,d\>, [\vec 0, \vec 1])\vee (u,v).\]
Then
\[ B \models\bigwedge_{i=1}^{k}u_{i}(b,d,\vec 0)\vee v = u_{i}(b,d,\vec 1)\vee v.\]
But the above equations in combination with~(\ref{eq:34}) produce
\[b \o v = d \o v\] 
and hence
\[\<a,b\>\vee (u,v)= \<a,d\>\vee (u,v).\]

Now suppose
\[A\times B \models \Phi\bigl(\<a,b\>, \<c,d\>, [\vec 0, \vec
  1]\bigr).\]
The reader can check that considering
$u=(a,\bigvee_{i=1}^{k}u_{i}(b,d,\vec 0)\o u_{i}(b,d,\vec 1))$  it
can be proved $a \o c = a$, and similarly with
$u=(c,\bigvee_{i=1}^{k}u_{i}(b,d,\vec 0)\o u_{i}(b,d,\vec 1))$ and $a \o c =c$, hence $a=c$.
\end{proof}

The following example will show that the complexity of formula $\Phi$ in the above
proposition cannot be improved for the general case.
\begin{prop}
Let $\V^{\o}$ the variety  with language $\{+,*,0,1, \o \}$ defined by the
axioms of $\V_L$ plus identities saying that $\o$ is a semilattice
operation for which $0 \o 1 = 0$. Then there exists neither a positive
nor an existential formula satisfying  (4) of 
Theorem~\ref{th:principal} for $\V^{\o}$.
\end{prop}

\begin{proof}
Define join-semilattice operations on $L_2$, $L_4$  and $L_5$ 
such that they are totally ordered with the order given by $0 > 1 > 2
> 3 > 4$.
Suppose that $\Phi$ satisfies (4) of Theorem~\ref{th:principal}, and
consider $L_5\times L_2$. 
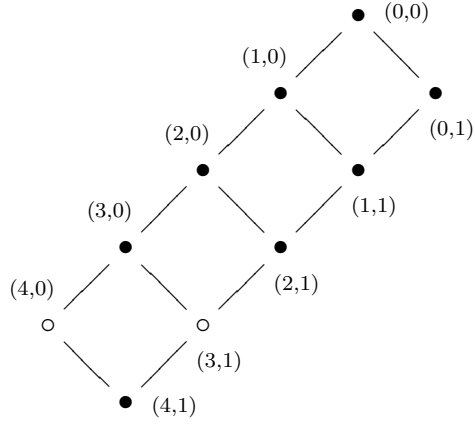
\begin{figure}[h]
\centerline{
\xymatrix@dr{
 {\bullet} \ar@{-}[r]^<{(0,0)}  \ar@{-}[d]_>{(1,0)} &  {\bullet}  \ar@{-}[d]^<{(0,1)} \\
 {\bullet} \ar@{-}[r]  \ar@{-}[d]_>{(2,0)} &  {\bullet}  \ar@{-}[d]^<{(1,1)} \\
 {\bullet} \ar@{-}[r]  \ar@{-}[d]_>{(3,0)} &  {\bullet}  \ar@{-}[d]^<{(2,1)} \\
 {\bullet} \ar@{-}[r]  \ar@{-}[d]_>{(4,0)} &  {\circ}  \ar@{-}[d]^<{(3,1)}^>{(4,1)}   \\
 {\circ} \ar@{-}[r]   &  {\bullet} 
}
}
\caption{The semilattice structure of $L_5\times L_2$.}\label{fig:L}
\end{figure}

The shaded dots in Figure~\ref{fig:L} form a subalgebra of $L_5\times L_2$; call it
$L$. The reader may check that $F: L_4\times L_2 \func L$, where
\[F(x) := \begin{cases} x & x\neq (3,1) \\ (4,1) & x = (3,1)
\end{cases}\]
is an isomorphism. Since 
$\Phi$ witnesses (4) of Theorem~\ref{th:principal} for $\V^{\o}$, we have
\[L_4 \times L_2 \models \Phi((3,0),(3,1),(0,1),(1,0)).\]
Applying $F$ everywhere,
\[L\models \Phi((3,0),(4,1),(0,1),(1,0)).\]
If we had $\Phi$ an existential formula, we  would obtain
\[L_5\times L_2\models \Phi((3,0),(4,1),(0,1),(1,0)),\]
since $L$ is a subalgebra of  $L_5\times
L_2$. We would then  conclude  $3 = 4$, an absurdity.

The fact that  there is no  positive formula $\Phi$ 
is an immediate consequence of the next two claims. We say that
$\V$ has \emph{compact factor congruences} if
every factor congruence of every algebra in $\mathcal{V}$ is compact.

\begin{claim}
$\V^{\o}$ has not compact factor congruences.
\end{claim}
\begin{proof}
If $\V^\o$ had compact factor congruences, there would exist  $\cero_{1}(w),\dots ,\cero_{N}(w)$, $%
\uno_{1}(w),$ $\dots,\uno_{N}(w)$ such that for every algebra $A=A_{1}\times
A_{2}\in \mathcal{V}$, $(\lambda _{1},\lambda _{2})\in A,$%
\begin{align*}
\ker \pi _{1} &=\Cg^{A} \bigl( [\vec{\cero}(\lambda 
_{1}),\vec{\cero}(\lambda
_{2})],[\vec{\cero}(\lambda _{1}),\vec{\uno}(\lambda _{2})]\bigr) \\
\ker \pi _{2} &=\Cg^{A} \bigl( [\vec{\uno}(\lambda 
_{1}),\vec{\uno}(\lambda
_{2})],[\vec{\cero}(\lambda _{1}),\vec{\uno}(\lambda _{2})]\bigr),
\end{align*}
by Lemma 4 in~\cite{9}. Since the language contains constants, we can
replace these new $\cero$'s and $\uno$'s by closed terms, and hence
\begin{align*}
\ker \pi _{1} &=\Cg^{A} \bigl([\vec{\cero},\vec{\cero}],[\vec{\cero},\vec{\uno}]\bigr) \\
\ker \pi _{2} &=\Cg^{A}\bigl( [\vec{\uno},\vec{\uno}],[\vec{\cero},\vec{\uno}]\bigr).
\end{align*}
Now, checking the axioms of $\V^\o$, we conclude that for every closed term
$t$ in the language of $\V^\o$, $t$ is equivalent to $0$ or $t$ is
equivalent to $1$ over $\V^\o$, hence we should have
\[\ker \pi _{1} =\Cg^{A}\bigl((0,0), (0,1)\bigl) \,\o\, \Cg^{A}\bigl((1,0),
(1,1)\bigl).\]
But the reader may check that if  we take
$A=L_5\times L_2$, the equivalence relation depicted in
Figure~\ref{fig:congr} is a congruence that contains the right hand side of the last equality,
and is clearly different from $\ker \pi _{1}$.
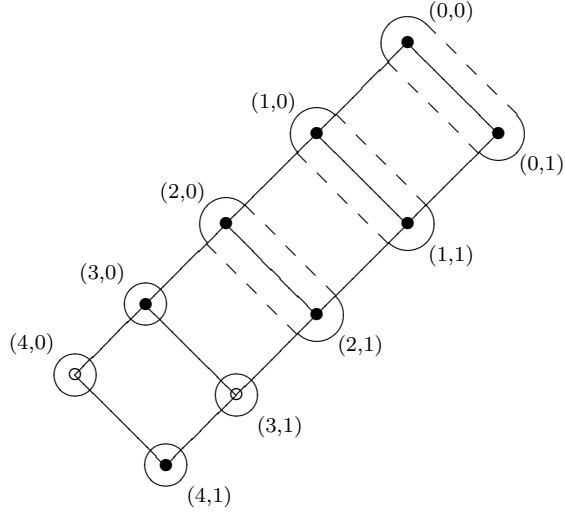
\begin{figure}[h]
\centerline{\xymatrix@dr@=2ex@L=1.5ex{
{\extrNEpi} \bordesup \bordeinf \alcentro{r}^<{(0,0)}  \alcentro{d}_>{(1,0)} &
{\extrSWpi}    \alcentro{d}^<{(0,1)} \\ 
{\extrNEpi} \bordesup \bordeinf \alcentro{r}  \alcentro{d}_>{(2,0)} &
{\extrSWpi}   \alcentro{d}^<{(1,1)} \\ 
{\extrNEpi} \bordesup \bordeinf \alcentro{r}  \alcentro{d}_>{(3,0)} &
{\extrSWpi}  \alcentro{d}^<{(2,1)} \\ 
  \sing  \alcentro{r}  \alcentro{d}_>{(4,0)} &   \singul{\circ}  \alcentro{d}^<{(3,1)}^>{(4,1)} \\
 \singul{\circ} \alcentro{r}   & \sing
}
}
\caption{One congruence in $L_5\times L_2$.}\label{fig:congr}
\end{figure}
\end{proof}
\begin{claim}
  Let $\V$ be a variety with $\vec 0$ \& $\vec 1$. Suppose that $\Phi$ is a positive
  formula that satisfies  (4) of 
  Theorem~\ref{th:principal}. Then    $\V$ has compact factor congruences.
\end{claim}
\begin{proof}
Let $A\in\V$. We will prove that if $\phi \romb
\phis = \0$, $\vec 0
\,\phi\,\vec e\,\phis\,\vec 1$ and $\vec 1\,\phi\,\vec f\,\phis\,\vec 0$, then
$\phi=\Cg(\vec 0,\vec e)
\o \Cg(\vec 1,\vec f)$, and hence is
compact.  Call $\th :=Cg(\vec 0,\vec e)
\o \Cg(\vec 1,\vec f)$. Trivially, $\th\subseteq\phi$. Assume
   $x \,\phi\, y$; by (4) of Theorem~\ref{th:principal}, we obtain
   $A\models \Phi(x,y,\vec e,\vec f)$. Since $\Phi$ is positive, it is preserved
   by homomorphic images and then 
\[A/\th\models
\Phi(x/\th,y/\th, \vec e/\th,\vec f/\th).\]
   Equivalently, 
   \[A/\th\models
   \Phi(x/\th,y/\th,\vec 0/\th,\vec 1/\th),\]
   and we obtain $x/\th=y/\th$. This implies $(x,y) \in \th$, hence
   $\phi\subseteq\th$,  and we have the result.
\end{proof}
\end{proof}

\section{Final Considerations}
We wish to mention that Theorem~\ref{th:principal} is in some sense a
consequence of Beth's definability theorem of first-order logic. This
is because the first evidence that the Determining Property implies
Definable Factor Congruences was obtained by a simple application of
 this theorem. Once we are sure about the concrete
 existence of certain first-order formula\footnote{As it's well known,
 formulas given by Beth's theorem can be effectively found.} we know that all efforts
 dedicated to find it are not sterile. Due to the expressive power of
 terms in algebra, it is common 
 that first-order definitions obtained in this area are more manageable
 that in the general case of model theory, and this fact makes the
 task for searching them a little easier.   

Another tool that was important for our research for this paper but it
was not included in the final presentation is the
Pierce sheaf. Working on this construction, we realized  that there was some sense in
developing a structural study of central elements when factor
congruences are not compact. 

Finally, we are indebted to Ross Willard for the
terms $s_i(\dots)$ and  $t_i(\dots)$, which appear in his
(unpublished) version of a Mal'cev condition for Boolean Factor Congruences.

\appendix
\section*{Appendix: A Preservation Result}\label{sec:preservation-result}
For the rest of the section, $N$ will be an
even natural number.

\begin{lemma}\label{l:preserv_dfc}
For  every word  $\alpha$ in the alphabet $\{1,\dots,N\}$ of
length no greater than $N$, let $\tau_\alpha = \tau_\alpha (x,y,\vec z,x_1,y_1,\dots,x_n,y_n)$ be
a formula preserved by direct products and by taking direct factors. Define:
\[
E_{m}:=\bigwedge_{\begin{subarray}{c}
m\leq\largo{\alpha}\leq N  \\ \largo{\alpha}\text{ even }\end{subarray}}
\ \Bigl(\sideset{}{_{\gamma\neq\varepsilon}}\bigwedge \tau_{\alpha\gamma} \limp
\tau_\alpha \Bigr)\qquad  O_{m}:= \bigwedge_{\begin{subarray}{c}
m\leq\largo{\alpha}\leq N  \\ \largo{\alpha}\text{ odd }\end{subarray}}
\ \Bigl(\sideset{}{_{\gamma\neq\varepsilon}}\bigwedge \tau_{\alpha\gamma} \limp
\tau_\alpha \Bigr).\]
Then,
\begin{enumerate}
\item\label{item:baja_O} For $2\leq m \leq N$, $m$ even, if $\bigl( \exists y_1 \forall x_1
\dots \exists y_n \forall x_n \; E_m\bigr) \y \bigl(\exists x_1 \forall y_1 \dots
\exists x_n \forall y_n\; O_{m+1}\bigr)$ is preserved by direct factors, so is
\[\bigl(\exists y_1 \forall x_1 
\dots \exists y_n\forall x_n \; E_m\bigr) \y \bigl(\exists x_1 \forall y_1 \dots
\exists x_n \forall y_n\; O_{m-1}\bigr).\]
\item\label{item:baja_E}  For $4\leq m \leq N$, $m$ even, if  $\bigl( \exists y_1 \forall x_1
\dots \exists y_n \forall x_n \; E_m\bigr) \y \bigl(\exists x_1 \forall y_1 \dots
\exists x_n \forall y_n\; O_{m-1}\bigr)$
is preserved by direct factors, so is  
\[\bigl(\exists y_1 \forall x_1 
\dots \exists y_n \forall x_n \; E_{m-2}\bigr) \y \bigl(\exists x_1 \forall y_1 \dots
\exists x_n \forall y_n\; O_{m-1}\bigr).\]
\end{enumerate}
\end{lemma}

Note that every subindex varies over words of length less than or
equal to $N$, so an expression of the form
``$\bigwedge_{\gamma\neq\varepsilon} \tau_{\alpha\gamma}$'' should be
read like ``$\bigwedge \{ \tau_{\alpha\gamma} :\gamma\neq\varepsilon
\text{ and } \largo{\alpha\gamma}\leq N \}$''. Therefore, if
$m\geq N$, $O_{m} = true$ (empty conjunction) and $E_N = 
\bigwedge_{\largo{\beta}=N} \tau_\beta$. Also, recall that the $i$-th
component of an element $a$ in a direct product $\Pi_i A_i$ is called $a^i$.

We will now state and prove two lemmas that will be helpful in order
to prove Lemma~\ref{l:preserv_dfc}. In the following  we will
assume that the tuple $\vec z$ has length equal to 1, since proofs
are exactly the same and this
simplification makes them easier to read.
\begin{lemma}\label{l:tech_E} 
Let $m$ be an even integer, $A_0, A_1\in \V$
and $c,d,e,a_1,\dots,a_{2n}\in A_0\times A_1$ such that $2\leq m\leq
N$, $A_0\times A_1 \models  E_{m} (c,d,e,a_1,\dots,a_{2n})$ and 
$ A_1 \models O_{m+1}(c^1,d^1,e^1,a_1^1,\dots,a_{2n}^1)$. Then $ A_0
\models E_{m}(c^0,d^0,e^0,a_1^0,\dots,a_{2n}^0)$ and if $\alpha$ has
length $m$ then 
\begin{equation}\label{eq:32}
A_0 \models \Bigl(\bigwedge_{\gamma\neq\varepsilon}
\tau_{\alpha\gamma}\Bigr)(c^0,d^0,e^0,a_1^0,\dots,a_{2n}^0) \ \ent \
A_0 \times A_1 \models
\Bigl(\bigwedge_{\mu}\tau_{\alpha\mu}\Bigr)(c,d,e,a_1,\dots,a_{2n}).
\end{equation}
\end{lemma}
\begin{proof}
By induction on $m$. If $m=N$, the first part is immediate since $E_N$ is a
conjunction of  formulas  preserved by taking direct factors and hence preserved by taking direct factors. The
second part is contained in the hypothesis.

To make the proof more readable, we will omit the string of
parameters. Take an even $m$ such that $2\leq m < N$ and suppose the lemma is proved for
$m+2$.  Assume 
\begin{align}
 A_0\times A_1 &\models  E_{m}; &&\qquad\text{Note that:} &E_{m} &= E_{m+2} \y  \bigwedge_{\largo{\alpha}= m}
 \; \Bigl(\sideset{}{_{\gamma\neq\varepsilon}}\bigwedge \tau_{\alpha\gamma} \limp
 \tau_\alpha \Bigr)\label{eq:9} \\
 A_1 &\models O_{m+1} &&  &O_{m+1} &= O_{m+3} \y  \bigwedge_{\largo{\alpha}= m+1}
 \; \Bigl(\sideset{}{_{\gamma\neq\varepsilon}}\bigwedge \tau_{\alpha\gamma} \limp
 \tau_\alpha \Bigr).\label{eq:10}
  \end{align}
By the first part of the inductive hypothesis we thus have $A_0\models E_{m+2}$. We have to
 see that $A_0 \models \bigwedge_{\largo{\alpha}= m}
 \; \bigl(\sideset{}{_{\gamma\neq\varepsilon}}\bigwedge \tau_{\alpha\gamma} \limp
 \tau_\alpha \bigr)$. Suppose now that for some $\alpha$ of length $m$,
 \begin{equation}
A_0 \models \sideset{}{_{\gamma\neq\varepsilon}}\bigwedge
 \tau_{\alpha\gamma}.\label{eq:11}
\end{equation}
In particular, for each $i,j\leq N$ we have
\[A_0 \models \sideset{}{_{\gamma\neq\varepsilon}}\bigwedge
 \tau_{\alpha i j \gamma}. \]
We will prove that $A_0 \models \tau_\alpha$ and the second part of
the lemma will be proved along the way.

By the second part of the inductive hypothesis 
(i.e., (\ref{eq:32})) we obtain, for all $j$,
\[A_0 \times A_1 \models \sideset{}{_\mu}\bigwedge
\tau_{\alpha i j \mu} \quad \text{or, in other symbols,} \quad A_0 \times A_1 \models \sideset{}{_{\gamma\neq\varepsilon}}\bigwedge
\tau_{\alpha i \gamma}. \]
Since this last formula is preserved by taking direct factors, we have 
\begin{equation}\label{eq:30}
A_1 \models 
\sideset{}{_{\gamma\neq\varepsilon}}\bigwedge \tau_{\alpha i \gamma}.
\end{equation}
Using~(\ref{eq:10})  (note $\largo{\alpha i} = m+1$), we have 
$ A_1 \models \tau_{\alpha i}$ for all $i$. This, together
with~(\ref{eq:30}) yields $A_1
\models\bigwedge_{\gamma\neq\varepsilon} 
\tau_{\alpha\gamma}$. Now we apply~(\ref{eq:11}), thus 
obtaining
\[A_0 \times A_1 \models \sideset{}{_{\gamma\neq\varepsilon}}\bigwedge
 \tau_{\alpha\gamma}.\]
Applying~(\ref{eq:9}),
\[A_0 \times A_1 \models  \tau_{\alpha}. \]
The last two formulas  jointly say 
\begin{equation}\label{eq:17}
A_0 \times A_1 \models
\Bigl(\bigwedge_{\mu}\tau_{\alpha\mu}\Bigr).
\end{equation}
We have proved (\ref{eq:11})$\ent$(\ref{eq:17}), which is the second
conclusion. Since $\tau_\alpha$ is 
preserved by taking direct factors, we obtain $A_0 \models  \tau_{\alpha},$
which is  the first conclusion. 
\end{proof}

\begin{lemma}\label{l:tech_O} 
Let $m$ be an even integer, $A_0, A_1\in \V$
and $c,d,e,a_1,\dots,a_{2n}\in A_0\times A_1$ such that $2\leq m\leq
N$, $A_0\times A_1 \models  O_{m-1}(c,d,e,a_1,\dots,a_{2n}) $ and 
$A_1 \models E_{m}(c^1,c^1,e^1,a_1^1,\dots,a_{2n}^1)$. Then $ A_0
\models O_{m-1}(c^0,d^0,e^0,a_1^0,\dots,a_{2n}^0)$ and if $\alpha$ has
length $m-1$ then 
\[ A_0 \models \Bigl(\bigwedge_{\gamma\neq\varepsilon}
\tau_{\alpha\gamma}\Bigr)(c^0,d^0,e^0,a_1^0,\dots,a_{2n}^0) \ \ent \ A_0 \times A_1 \models  \Bigl(\bigwedge_{\mu}\tau_{\alpha\mu}\Bigr)(c,d,e,a_1,\dots,a_{2n}).\]
\end{lemma}
\begin{proof}
By induction on $m$. If $m=N$, the hypotheses are:
\begin{align}
 A_0\times A_1 &\models  O_{N-1} && =  \bigwedge_{\largo{\alpha}= N-1}
 \; \Bigl(\sideset{}{_{\gamma\neq\varepsilon}}\bigwedge \tau_{\alpha\gamma} \limp
 \tau_\alpha \Bigr)\label{eq:12} \\
 A_1 &\models E_{N} && = \bigwedge_{\largo{\beta}= N}
 \tau_\beta.\label{eq:13}
\end{align}	

Assume that for some $\alpha$ of length $N-1$,
 \begin{equation*}
A_0 \models \sideset{}{_{\gamma\neq\varepsilon}}\bigwedge
 \tau_{\alpha\gamma} \; = \; \bigwedge_i \tau_{\alpha i},\label{eq:14}
\end{equation*}
By using~(\ref{eq:13}) and preservation by direct products, we have
\[A_0 \times A_1 \models  \bigwedge \tau_{\alpha i}. \]
By~(\ref{eq:12}) we have $A_0 \times A_1 \models  \tau_{\alpha}$, thus
obtaining
\[A_0 \times A_1 \models  \bigwedge_\mu \tau_{\alpha \mu}. \]
We have proved the second part of the lemma. Passing to
$A_0$ we obtain the first part.

Now take an even $m$ such that $2\leq m < N$ and suppose the lemma is proved for
$m+2$.  Assume 
\begin{align*}
 A_0\times A_1 &\models  O_{m-1} && =  O_{m+1} \ \y  \bigwedge_{\largo{\alpha}= m-1}
 \; \Bigl(\sideset{}{_{\gamma\neq\varepsilon}}\bigwedge \tau_{\alpha\gamma} \limp
 \tau_\alpha \Bigr)\\ 
 A_1 &\models E_{m} && = E_{m+2} \ \y  \bigwedge_{\largo{\alpha}= m}
 \; \Bigl(\sideset{}{_{\gamma\neq\varepsilon}}\bigwedge \tau_{\alpha\gamma} \limp
 \tau_\alpha \Bigr).
  \end{align*}
By inductive hypothesis we thus have $A_0\models O_{m+1}$. The rest
 of the argument closely parallels the proof of Lemma~\ref{l:tech_E}.
\end{proof}
\begin{proof}[Proof of Lemma~\ref{l:preserv_dfc}]
To see (\ref{item:baja_O}), 
 suppose  
\[A_0\times A_1 \models \Bigl(\bigl(\exists y_1\forall x_1 
\dots \exists y_n\forall x_n \; E_m\bigr) \y \bigl(\exists x_1 \forall y_1 \dots
\exists x_n \forall y_n\; O_{m-1}\bigr)\Bigr)(c,d,e).\]
Thus we have
(Skolem) functions $G_1,\dots,G_n$ such that $G_i $ is $(i-1)$-ary and
\begin{align*}
A_0\times A_1 &\models \forall \vec y O_{m-1} (c,d,e,G_1,y_1,\dots,G_n(y_1,\dots,y_{n-1}),y_n).
\end{align*}
Since $O_{m-1}$ implies $O_{m+1}$, we have
\[A_0\times A_1 \models \Bigl(\bigl(\exists y_1\forall x_1 
\dots \exists y_n\forall x_n \; E_m\bigr) \y \bigl(\exists x_1 \forall y_1 \dots
\exists x_n \forall y_n\; O_{m+1}\bigr)\Bigr)(c,d,e).\]
And then, since this formula is preserved by taking direct factors, by hypothesis,
\begin{align}
  A_0&\models \bigl(\exists y_1\forall x_1 
    \dots \exists y_n\forall x_n \;
    E_m\bigr)(c^0,d^0,e^0)\label{eq:35} \\
  A_1 &\models \bigl(\exists y_1\forall x_1 
    \dots \exists y_n\forall x_n \; E_m\bigr)(c^1,d^1,e^1).
\end{align}
Thus we have functions $F_1,\dots,F_n$ such that
\begin{equation*}
A_1 \models \forall \vec x E_m
(c^1,d^1,e^1,x_1,F_1,\dots,x_n,F_n(x_1,\dots,x_{n-1})).
\end{equation*}
Now, for $j=1,\dots,n$  define $j$-ary functions $p_j$ from $A_0$
to $A_0\times A_1$, $p_j = p_j(a_1,\dots,a_{j})$: 
\begin{align*}
 p_1&:= \< a_1, F_1\> \\
 p_2 &:= \<a_2, F_2( G_1^1)\> \\
 p_j &:=  \<a_j, F_j\bigl( G_1^1, G_2(p_1)^1,\dots,
 G_{j-1}(p_1,\dots,p_{j-2})^1\bigr) \>.
\end{align*}
The reader may check that this selection ensures, for each $\vec a \in
A_0^n$:
\begin{align*}
A_0\times A_1 &\models O_{m-1}
\bigl(c,d,e,G_1,p_1,\dots,G_n(p_1,\dots,p_{n-1}),p_n\bigr) \\ 
A_1 &\models E_m
\bigl(c^1,d^1,e^1,G_1^1,p_1^1,\dots,G_n(p_1,\dots,p_{n-1})^1,
p_n^1\bigr).
\end{align*}
We may apply Lemma~\ref{l:tech_O} and obtain
\begin{equation*}
A_0 \models O_{m-1}
\bigl(c^0,d^0,e^0,G_1^0,p_1^0,\dots,G_n(p_1,\dots,p_{n-1})^0,p_n^0\bigr).
\end{equation*}
Equivalently,
\begin{equation*}
A_0 \models O_{m-1}
\bigl(c^0,d^0,e^0,G_1^0,a_1,G_2(p_1)^0,a_2,\dots,G_n(p_1,\dots,p_{n-1})^0,a_n\bigr).
\end{equation*}
Now, defining $H_j:A_0^{j-1} \func A_0$ as follows:
\begin{align*}
H_1 &:=  G_1^0 \\
H_2(y_1) &:= G_2(p_1(y_1))^0 \\
H_j(y_1,\dots,y_{j-1}) &:= G_j(p_1(y_1),\dots,p_{j-1}(y_1,\dots,y_{j-1}))^0,
\end{align*}
we see at once that
\begin{equation*}
A_0 \models\forall \vec y O_{m-1}
\bigl(c^0,d^0,e^0,H_1,y_1,\dots,H_n(y_1,\dots,y_{n-1}),y_n\bigr),
\end{equation*}
and then
\[A_0  \models  \bigl(\exists x_1 \forall y_1 \dots
\exists x_n \forall y_n\; O_{m-1}\bigr)(c^0,d^0,e^0).\]
This, together with~(\ref{eq:35}), proves this case.
\smallskip

Part (\ref{item:baja_E}),  is entirely analogous to the former, and it's proved by
using Lemma~\ref{l:tech_E}.
\end{proof}

The  hypotheses of the next theorem are the same as in
Lemma~\ref{l:preserv_dfc}, and we repeat then for the ease of reference.

\begin{theorem}\label{t:preserv_ker}
For  every word  $\alpha$ in the alphabet $\{1,\dots,N\}$ of
length no greater than $N$, let $\tau_\alpha = \tau_\alpha (x,y,\vec z,x_1,y_1,\dots,x_n,y_n)$ be
a formula preserved by taking direct products and  direct factors. Define:
\[
E_{m}:=\bigwedge_{\begin{subarray}{c}
m\leq\largo{\alpha}\leq N  \\ \largo{\alpha}\text{ even }\end{subarray}}
\ \Bigl(\sideset{}{_{\gamma\neq\varepsilon}}\bigwedge \tau_{\alpha\gamma} \limp
\tau_\alpha \Bigr)\qquad  O_{m}:= \bigwedge_{\begin{subarray}{c}
m\leq\largo{\alpha}\leq N  \\ \largo{\alpha}\text{ odd }\end{subarray}}
\ \Bigl(\sideset{}{_{\gamma\neq\varepsilon}}\bigwedge \tau_{\alpha\gamma} \limp
\tau_\alpha \Bigr).\]
 Then the formula 
\begin{equation}\label{eq:31}
\bigl( \exists y_1 \forall x_1
\dots \exists y_n \forall x_n \; E_2\bigr) \y \bigl(\exists x_1 \forall y_1 \dots
\exists x_n \forall y_n\; O_1\bigr)
\end{equation}
is  preserved by taking direct factors and direct products.
\end{theorem}
\begin{proof}
First observe that
\[\bigl( \exists y_1 \forall x_1
\dots \exists y_n \forall x_n \; E_N\bigr) \y \bigl(\exists x_1 \forall y_1 \dots
\exists x_n \forall y_n\; O_{N+1}\bigr) = \exists y_1 \forall x_1
\dots \exists y_n \forall x_n \;  \bigwedge_{\largo{\beta}=N} \tau_\beta \]
is preserved by direct factors. This is immediate since
conjunction and quantification of  formulas preserved by taking direct factors is
again preserved by taking direct factors. Successive application of
Lemma~\ref{l:preserv_dfc} yields that~(\ref{eq:31}) is preserved by
taking direct factors.

The proof that~(\ref{eq:31}) is preserved by direct products is a
straightforward calculation.
\end{proof}

\bigskip
\bigskip

\begin{quote}
CIEM --- Facultad de Matem\'atica, Astronom\'{\i}a y F\'{\i}sica 
(Fa.M.A.F.) 

Universidad Nacional de C\'ordoba - Ciudad Universitaria

C\'ordoba 5000. Argentina.

\texttt{sterraf@mate.uncor.edu}

\texttt{vaggione@mate.uncor.edu}
\end{quote}
\end{document}